\documentclass[a4paper,12pt]{amsart}

\usepackage{amsmath,amssymb,amsthm,graphicx,pstricks,comment}
\usepackage[all]{xy}

\usepackage{lscape}
\usepackage{hyperref}
\usepackage{tikz}
\usetikzlibrary{arrows}
\usetikzlibrary{snakes}
\usetikzlibrary{shapes}

\author{Claire Amiot}
\address{Institut Fourier, 100 rue des maths, 38402 Saint Martin d'H\`eres}
\email{claire.amiot@ujf-grenoble.fr}

\numberwithin{equation}{subsection}
\numberwithin{figure}{section}

\newtheorem{theorem}{Theorem}[section]

\newtheorem{lemma}[theorem]{Lemma}
\newtheorem{corollary}[theorem]{Corollary}
\newtheorem{proposition}[theorem]{Proposition}
\theoremstyle{remark}
\newtheorem{remark}[theorem]{Remark}

\theoremstyle{definition}
\newtheorem{definition}[theorem]{Definition}

\newcommand{\intS}{S\backslash \partial S}

\DeclareMathOperator{\Jac}{J}

\DeclareMathOperator{\MCG}{\mathcal{M}}
\DeclareMathOperator{\Homeo}{Homeo}
\renewcommand{\mod}{\operatorname{mod}}


\title[The derived category of surface algebras]{The derived category of surface algebras: the case of the torus with one boundary component}

\begin{document}
\maketitle

\begin{abstract}
In this paper we refine the main result of a previous paper of the author with Grimeland on derived invariants of surface algebras. We restrict to the case where the surface is a torus with one boundary component and  give an easily computable derived invariant for such surface algebras. This result permits to give answers to open questions on gentle algebras: it provides examples of gentle algebras with the same AG-invariant (in the sense of Avella-Alaminos and Geiss) that are not derived equivalent and gives a partial positive answer to a conjecture due to Bobi\'nski and Malicki on gentle $2$-cycle algebras.

\end{abstract}

\tableofcontents
\section{Introduction}

Let $k$ be a field and $\Lambda$ be a finite dimensional $k$-algebra. The bounded derived category $\mathcal{D}^b(\mod \Lambda)$ is an essential tool in representation theory: this triangulated category encodes homological information about the algebra $\Lambda$ and about its module category. Therefore it is of great interest to determine when two given algebras are derived equivalent (i.e. have equivalent derived categories). This question has lead to one very important tool in representation theory: tilting theory, and has been studied a lot since the 80's (see \cite{Hap88}, or \cite{AHHK} for a broad overview of the subject).

In \cite{AG}, we focus on a very special class of algebras called surface algebras. These algebras are constructed from combinatorial datas: a triangulation $\Delta$ of an oriented compact surface $S$ with marked points $M$ on its boundary and an amdissible cut (see Definition \ref{def admissible cut}). From this data one defines a function $d:\pi_1^{\rm free}(S)\to \mathbb Z$, where $\pi_1^{\rm free}(S)$ is the set of loops on $S$ modulo free homotopy. The main result of \cite{AG} (Theorem~\ref{thmAG}) asserts that two algebras $\Lambda=(S,M,\Delta,d)$ and $\Lambda'=(S',M',\Delta',d')$ are derived equivalent if and only if there exists a homeomorphism $\Phi:S\to S'$ such that $d=d'\circ\Phi$ as functions on $\pi_1^{\rm free}(S)$. This result is very simple to state but not so easy to use in practice since for genus $\geq 1$ the mapping class group of $S$ and its action on $\pi_1^{\rm free}(S)$ are difficult to understand. 

In this paper, we restrict to the case of surface algebras coming from the torus with one boundary component. In that special case, we have a good understanding of the mapping class group of the surface and of its action on the degree map $d$. Using topological and combinatorial arguments we prove that $d=d'\circ\Phi$ if and only if $\gcd(d(a),d(b))=\gcd(d'(a),d'(b))$ where $a$ and $b$ are generators of the fundamental group of the surface. As a consequence, we obtain answers to open questions: we exhibit gentle algebras with the same AG-invariant (in the sense of \cite{AAG}) that are not derived equivalent (Corollary \ref{cor AG}). We also prove a part of a conjecture by Bobi\'nski and Malicki \cite{BM} on gentle $2$-cycles algebras (Proposition \ref{propgentle2}). 

\subsection*{Acknowledgment} The author thanks Fran\c cois Dahmani for helpful discussions on mapping class groups and Luis Paris for pointing out a mistake in a previous version (see Remark \ref{remark Paris}), and the reference \cite{PR}.

\section{Preliminaries}

In this section we recall the definition of surface algebras, and give some basic topological properties that are needed in the proof of the main result. 

\subsection{Surface algebras}

Let $S$ be an oriented connected closed compact surface of genus $g$, with $b$ boundary components and $M$ be a finite set of marked points on the boundary of $S$, such that there is at least one marked point on each component of the boundary.  By an \emph{arc}, we mean the isotopy class of a simple curve in the interior of $S$ with endpoints being in $M$ that does not cut out a monogon or a digon of the surface. A pair of arcs is called \emph{admissible} if there exist representants in their respective isotopy class that do not intersect in the interior of $S$. A \emph{triangulation} of $(S,M)$ is a maximal collection of pairwise admissible arcs.

Fomin, Shapiro and Thurston associate a quiver $Q^\Delta$ (via signed adjacency matrices \cite[Definition 4.1]{FST}) to each triangulation $\Delta$ as follows: the set of vertices $Q^\Delta_0$ of $Q^\Delta$ is the set of arcs of $\Delta$, and the set of arrows $Q^\Delta_1$ is in bijection with the set of oriented internal angles between two arcs of $\Delta$. We also denote by $Q^\Delta_2$ the set of oriented internal triangles of $\Delta$, that are the triangles whose three sides are arcs (and not boundary segments) of the triangulation. For each such triangle $\tau\in \Delta$, there exist three arrows $\alpha_\tau$, $\beta_\tau$ and $\gamma_\tau$. The Jacobian algebra associated to $\Delta$ has been introduced in \cite{ABCP}. It is defined as $$ \Jac (\Delta):= kQ^\Delta/ \langle \alpha_\tau\beta_\tau,\ \beta_\tau\gamma_\tau,\ \gamma_\tau\alpha_\tau,\  \forall \tau\in Q^\Delta_2\rangle.$$

\begin{definition}\cite[Def. 2.8]{AG}\label{def admissible cut}
A \emph{degree-$1$ map} of $\Delta$ is a map $d:Q_1\to \mathbb Z$ such that for each triangle $\tau$ in $Q_2$, $d(\alpha_\tau)+d(\beta_\tau)+d(\gamma_\tau)=1$.
An \emph{admissible cut} of $\Delta$ is  a degree-$1$ map $d:Q_1\to \{0,1\}$ such that $d(\alpha)=0$ for each angle $\alpha$ in a triangle which is not internal. 

An admissible cut defines a positive $\mathbb Z$-grading on the Jacobian algebra $\Jac(\Delta)$. The \emph{surface algebra} (first introduced in \cite{DRS}) $\Lambda:=(\Delta,d)$  associated to $\Delta$ and $d$, is the degree zero subalgebra of $\Jac(\Delta,d)$. 
\end{definition}

Note that given a graded triangulation $(\Delta,d)$ (that is a triangulation with a degree-$1$ map $d$) and an arc $i$ of $\Delta$, on can flip it to obtain another graded triangulation denoted by $\mu_i^L(\Delta,d)$ (see \cite[Def 2.13]{AG}). 

\subsection{Fundamental group and free homotopy}

A \emph{loop} in $S$ is  a continuous map $[0,1]\to S\backslash \partial S$ with $f(0)=f(1)$. Two loops $f$ and $g$ are \emph{free homotopic} if there exists a continuous map $$h:[0,1]\times[0,1]\to \intS$$ such that  $h(0,-)=f$, $h(1,-)=g$ and $h(-,0)=h(-,1)$. We denote by $\pi_1^{\rm free}(S)$ the set of loops modulo free homotopy. Note that if $f$ is a loop in $S$ and $t_0\in [0,1]$, then the loop $g$ defined by 
\begin{equation}\label{free htp rmk} g(t)=\left\{\begin{array}{ll} f(t+t_0) & \textrm{if } 0\leq t\leq 1-t_0\\ f(t+t_0-1) & \textrm{if } 1-t_0\leq t\leq 1\end{array}\right. \end{equation}is free homotopic to $f$.

For $x_0$ a point in $S\backslash \partial S$, we denote by $\pi_1(S,x_0)$ the \emph{fundamental group} of $S$ with base point $x_0$, that is the set of  loops with base point $x_0$ modulo homotopy among loops based in $x_0$. Since the surface is connected the group $\pi_1(S,x_0)$ does not depend on the choice of $x_0$, so we will often denote it by $\pi_1(S)$. If two loops based in $x_0$ are homotopic, they obviously are free homotopic, so there is a natural map $\pi_1(S,x_0)\to\pi_1^{\rm free}(S)$. However the converse is not true, for instance if $f$ and $g$ are loops based in $x_0$, then $f*g$ is not homotopic to $g*f$ in general, but these two loops are free homotopic by (\ref{free htp rmk}). Hence the product of the group $\pi_1(S,x_0)$ is not well defined at the level of $\pi_1^{\rm free}(S)$, which does not have any group structure. In fact we have the following classical result (see 1.2.1 of \cite{FM}) for which we give a proof for the convenience of the reader.
\begin{lemma}
The set $\pi_1^{\rm free}(S)$ is in natural bijection with the set of conjugacy classes of $\pi_1(S,x_0)$. 
\end{lemma}
\begin{proof}
Let $\pi_1(S,x_0)\to \pi_1^{\rm free}(S)$ be the natural map. Then if $\alpha, \beta$ are loops on $S$ based on $x_0$, then $\alpha*\beta*\alpha^{-1}$ is free homotopic to $\beta*\alpha^{-1}*\alpha$ by the remark above, and hence is free homotopic to $\beta$. Hence this map yields a well defined map $\mathcal{C}\ell(\pi_1(S,x_0))\to \pi_1^{\rm free}(S)$. To see that it is onto, we define an inverse map. Let $\gamma\in\pi_1^{\rm free}(S)$, and let $f:[0,1]\to S$ be a continuous map with $f(0)=\gamma(0)$ and $f(1)=x_0$, it exists since $S$ is connected. Then define a map $\pi_1^{\rm free}(S)\to \mathcal{C}\ell (\pi_1(S,x_0))$ by $\gamma\mapsto {\mathcal C}\ell (f*\gamma*f^{-1})$. This map is well defined, indeed if $g$ is another curve on $S$ linking $\gamma(0)$ to $x_0$ then we have $$\mathcal{C}\ell (g*\gamma*g^{-1})=\mathcal{C}\ell((g*f^{-1})*(f*\gamma*f^{-1})*(g*f^{-1})^{-1})=\mathcal{C}\ell ((f*\gamma*f^{-1}))$$ since $g*f^{-1}$ is a loop on $S$ based on $x_0$.  
This map is clearly a right inverse for the map $\mathcal{C}\ell (\pi_1(S,x_0))\to \pi_1^{\rm free}(S)$. One only has to prove that it is a left inverse. Let $\gamma$ be a loop on $S$, and $f$ be a curve linking $\gamma(0)$ to $x_0$. Then the loop $f*\gamma*f^{-1}$ is free homotopic to $\gamma*f^{-1}*f$ by (\ref{free htp rmk}), and so free homotopic to $\gamma$, which ends the proof. 
\end{proof}

We denote by $H_1(S,\mathbb Z)$ the first group of the singular homology of $S$. There is a natural morphism of groups $\pi_1(S,x_0)\to H_1(S,\mathbb Z)$ which is the abelianization. By the lemma above, this map factors through  
 $\pi_1(S,x_0)\to \pi_1^{\rm free}(S)$ in a map $$[-]:\pi_1^{\rm free}(S)\to H_1(S,\mathbb Z).$$ 
\subsection{Deformation retract}
Let $T$ be triangle in $\mathbb R^2$ with counterclockwise orientation with vertices $A$, $B$ and $C$. For each triangle $\tau$ of $\Delta$ fix a contiuous map $\varphi_\tau:T\to \tau$ which preserves orientation and which is injective when restricted at $T\backslash \{A,B,C\}$. Define $A'$ (resp. $B'$ and $C'$) the middle of $[BC]$ (resp. $[AC]$, and $[AB]$), and $G$ a point in the interior of $T$. For each internal triangle $\tau$ of $\Delta$ define 
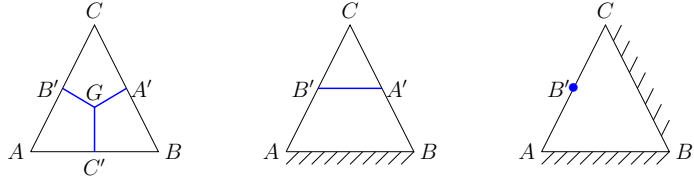
\[R_\tau:= \varphi_{\tau}([A'G]\cup[B'G]\cup[C'G]).\]

\[\scalebox{0.7}{\begin{tikzpicture}[scale=1.2,>=stealth]

\node at (8,0) [xshift=-8pt]{$A$};
\node at (9,2) [yshift=8pt]{$C$};
\node at (10,0) [xshift=8pt]{$B$};

\draw (8,0)--(9,2)--(10,0)--(8,0);
\draw[thick, blue] (8.5,1)--(9.5,1);
\node at (8.5,1) [xshift=-8pt]{$B'$};\node at (9.5,1) [xshift=8pt]{$A'$};

\draw (8,-0.2)--(8.2,0);
\draw (8.2,-0.2)--(8.4,0);\draw (8.4,-0.2)--(8.6,0);\draw (8.6,-0.2)--(8.8,0);\draw (8.8,-0.2)--(9,0);\draw (9,-0.2)--(9.2,0);
\draw (9.2,-0.2)--(9.4,0);\draw (9.4,-0.2)--(9.6,0);\draw (9.6,-0.2)--(9.8,0);\draw (9.8,-0.2)--(10,0);

\node at (4,0) [xshift=-8pt]{$A$};
\node at (5,2) [yshift=8pt]{$C$};
\node at (6,0) [xshift=8pt]{$B$};
\draw (4,0)--(5,2)--(6,0)--(4,0);
\draw[thick, blue] (4.5,1)--(5,0.7)--(5,0);
\draw[thick, blue] (5.5,1)--(5,0.7);
\node at (4.5,1) [xshift=-8pt]{$B'$};\node at (5.5,1) [xshift=8pt]{$A'$};
\node at (5,0) [yshift=-8pt]{$C'$};\node at (5,0.7)[yshift=8pt]{$G$};

\node at (12,0) [xshift=-8pt]{$A$};
\node at (13,2) [yshift=8pt]{$C$};
\node at (14,0) [xshift=8pt]{$B$};

\draw (12,0)--(13,2)--(14,0)--(12,0);
\node at (12.5,1) [xshift=-8pt]{$B'$};
\node at (12.5,1)[blue]{$\bullet$};

\draw (12,-0.2)--(12.2,0);
\draw (12.2,-0.2)--(12.4,0);\draw (12.4,-0.2)--(12.6,0);\draw (12.6,-0.2)--(12.8,0);\draw (12.8,-0.2)--(13,0);\draw (13,-0.2)--(13.2,0);
\draw (13.2,-0.2)--(13.4,0);\draw (13.4,-0.2)--(13.6,0);\draw (13.6,-0.2)--(13.8,0);\draw (13.8,-0.2)--(14,0);

\draw (13.25,2)--(13.125,1.75); \draw (13.375,1.75)--(13.25,1.5);\draw (13.5,1.5)--(13.375,1.25); \draw (13.625,1.25)--(13.5,1); \draw (13.75,1)--(13.625,0.75); \draw (13.875,0.75)--(13.75,0.5);\draw (14,0.5)--(13.875,0.25);

\end{tikzpicture}}\]

If $\tau$ is a triangle with exactly one side being a boundary segment, say $\varphi_\tau([AB])$, then define
\[ R_\tau:=\varphi_\tau([A'B']).\]

If $\tau$ is a triangle with exaclty two sides being boundary segments, say $\varphi_\tau([AB])$ and $\varphi_\tau([BC])$, then define
\[R_\tau:= \varphi_\tau(\{B'\}).\]

And define a graph $R$ on $\intS$ as $R=\bigcup_{\tau\in \Delta} R_\tau$.

\begin{lemma}\label{lemma retract}
There exists a map $r:S\backslash \partial S\to R$ that is a deformation retract of $S\backslash\partial S$. So $r$ induces an isomorphism of groups $\pi_1(S,x_0)\to\pi_1(R,r(x_0))$ and a bijection $\pi_1^{\rm free}(S)\to\pi_1^{\rm free}(R)$.
\end{lemma}

\begin{proof} 
For $x\in T\backslash([AB]\cup[BC])$, we define $r_1(x)=B'$. For $x\in T\backslash([AB]\cup\{C\})$ we define $r_2(x)$ as the interection point between the line $(Cx)$ and $[A'B']$. Then for $x\in T\backslash\{A,B,C\}$ we define 
\[r_3(x)=\left\{
\begin{array}{ll}  (Cx)\cap ([B'G]\cup[GA']) & \textrm{ if }x  \textrm{ is in the quadrilateral } B'CA'G;\\
(Ax)\cap ([B'G]\cup[GC']) & \textrm{ if }x  \in C'AB'G;\\
(Bx)\cap ([A'G]\cup[GC']) & \textrm{ if }x  \in A'BC'G;
\end{array}\right.\]

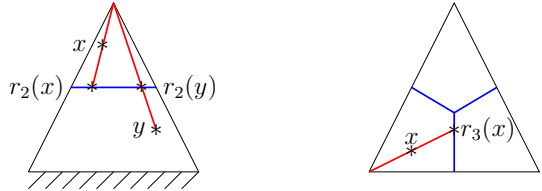
\[\scalebox{0.8}{\begin{tikzpicture}[scale=1.4,>=stealth]

\draw (0,0)--(1,2)--(2,0)--(0,0);
\draw[thick, blue] (0.5,1)--(1.5,1);

\draw (0,-0.2)--(0.2,0);
\draw (0.2,-0.2)--(0.4,0);\draw (0.4,-0.2)--(0.6,0);\draw (0.6,-0.2)--(0.8,0);\draw (0.8,-0.2)--(1,0);\draw (1,-0.2)--(1.2,0);
\draw (1.2,-0.2)--(1.4,0);\draw (1.4,-0.2)--(1.6,0);\draw (1.6,-0.2)--(1.8,0);\draw (1.8,-0.2)--(2,0);
\draw[thick, red] (1,2)--(0.75,1);
\draw[thick, red] (1,2)--(1.5,0.5);
\node at (0.87,1.5) {$*$};
\node at (0.6, 1.5) {$x$};
\node at (0.75,1) {$*$}; 
\node at (1.5,0.5) {$*$};
\node at (1.33,1) {$*$};
\node at (1.3,0.5) {$y$};
\node at (1.9,1) {$r_2(y)$};
\node at (0.1,1) {$r_2(x)$};

\draw (4,0)--(5,2)--(6,0)--(4,0);
\draw[thick, blue] (4.5,1)--(5,0.7)--(5,0);
\draw[thick, blue] (5.5,1)--(5,0.7);
\draw[thick, red] (4,0)--(5,0.5);
\node at (4.5,0.25) {$*$};
\node at (4.5,0.4) {$x$};
\node at (5,0.5) {$*$};
\node at (5.4,0.5) {$r_3(x)$};

\end{tikzpicture}}\]

Then for $x\in\intS$, if $x\in \tau$ where $\tau$ is a triangle of $\Delta$ we define $r(x):=\varphi_\tau\circ r_i\circ\varphi^{-1}_\tau (x)$ where $i$ is the number of sides of $\tau$ which are internal arcs of $\Delta$. It is an easy exercise to prove that $r$ is well defined map $\intS\to R$ and a deformation retract.

\end{proof}

\subsection{Intersection between a loop and an arc}\label{subsection intersection}
Let $i$ be an oriented arc of $(S,M)$.  For $\gamma$ a loop that cuts finitely and transversally $i$, we denote by ${\rm Int}^+(\gamma,i)$ (resp. ${\rm Int}^-(\gamma,i)$)  the number of times where $\gamma$ and $i$ intersect in such a way the intersection agrees (resp. disagrees) with the orientation of $S$. Define by 
${\rm Int}(\gamma,i)={\rm Int}^+(\gamma,i)-{\rm Int}^-(\gamma,i)$ the algebraic number of intersections between $\gamma$ and $i$. 

The next lemma is classical in the case of the intersection of two loops. We give here a proof for the convenience of the reader.

\begin{lemma}\label{lemma intersection}The map ${\rm Int}(-,i)$ induces a map $\pi_1^{\rm free}(S)\to \mathbb Z$ and  a $\mathbb Z$-linear map $H_1(S,\mathbb Z)\to \mathbb Z$ such that the following diagram commutes.
\[\xymatrix{\pi_1^{\rm free}(S)\ar[r]^{[-]}\ar[d] & H_1(S,\mathbb Z)\ar[dl]\\ \mathbb Z & }.\]
\end{lemma}

\begin{proof} Let $r:\intS\to R$ be the deformation retract defined above, and $\gamma$ be a loop on $S$ based on $x_0$. Then $r(\gamma)$ is a loop on $R$ based on $r(x_0)$ that meets $r(i)$ finitely many times and that does not make U-turn on $r(i)$ and we have ${\rm Int}(\gamma,i)={\rm Int}(r(\gamma),i)$. Moreover if two loops on $R$ are homotopic then one can pass from one to the other by adding and removing U-turns, so  they have the same algebraic number of intersection with $i$. Hence the map ${\rm Int}_R(-,r(i))$ yields well-defined and compatible maps $\pi_1(R,r(x_0))\to \mathbb Z$. Therefore by Lemma~\ref{lemma retract}, ${\rm Int}(-,i)$ induces a map $\pi_1(S,x_0)\to \mathbb Z$. Moreover, by definition ${\rm Int}(\gamma*\gamma',i)={\rm Int}(\gamma,i)+{\rm Int}(\gamma',i)$ so ${\rm Int}(-,i):\pi_1(S,x_0)\to \mathbb Z$ is a morphism of groups.  Since $\mathbb Z$ is abelian, this morphism is constant on conjugacy classes and factors through $H_1(S,\mathbb Z)\to \mathbb Z$ since $H_1(S,\mathbb Z)$ is the abelianization of $\pi_1(S,x_0)$.

\end{proof}

\subsection{Degree of a loop}\label{subsection degree}

In this section we show that any degree-$1$ map $d$ on a triangulation $\Delta$ defines a map $d:\pi_1^{\rm free}(S)\to \mathbb Z$ and describe the properties of this map.

Let $\gamma$ be a loop that cuts transversally each arc of $\Delta$ and does not cut the same arc twice in succession. For an arrow $\alpha:i\to j$ in a triangle $\tau\in \Delta$ we say that $\gamma$ \emph{intersects  positively (resp. negatively)} $\alpha$  if $\gamma$ intersects the arc $i$ (resp. $j$), then the arc $j$  (resp. $i$) and stay in the interior of $\tau$ in between. 

\begin{lemma}There is a well defined map $\overline{(-)}^\Delta:\pi_1^{\rm free}(S)\to \mathbb Z Q_1$ such that if $\gamma$ is a closed curve that cuts transversally each arc of $\Delta$ and not the same arc twice in succession then $\overline{\gamma}^\Delta=\sum_{\alpha}(n_\alpha^+-n_\alpha^-) \alpha$ where $n_\alpha^+$ (resp. $n_\alpha^-$) is the number of times  where $\gamma$ intersects positively (resp. negatively) the arrow $\alpha$.
\end{lemma}

\begin{proof}
The free homotopy class of a loop on a graph is determined by a word $e_1^{\pm}\ldots e_n^{\pm}$ where the $e_i$ are edges of the graph modulo the relation $e^+e^-=e^-e^+=\emptyset$ and cyclic equivalence. Hence it is clear how to associate an element in $\mathbb Z Q_1$ to an element in $\pi_1^{\rm free}(R)$. We conclude by Lemma \ref{lemma retract}.
\end{proof}

\begin{definition}
Let $d$ be a degree-$1$ map on $\Delta$. Then we set $d(\gamma)=d(\overline{\gamma}^\Delta)$ which defines a map $d:\pi_1^{\rm free}(S)\to \mathbb Z$.
\end{definition}

Note that $d$ does not factor through $[-]:\pi_1^{\rm free}(S)\to H_1(S,\mathbb Z)$ in general (cf Remark \ref{remark}).  However we have the following result.

\begin{lemma}\label{lemma d-d' homology}
Let $d$ and $d'$ be degree-$1$ maps $\mathbb Z Q_1\to \mathbb Z$. Then the map $d-d':\pi_1^{\rm free}(S)\to \mathbb Z$ factors through $[-]$ and the induced map $[d-d']:H_1(S,\mathbb Z)\to \mathbb Z$ is $\mathbb Z$-linear.
\end{lemma}

\begin{proof}

Let consider the following complex of $\mathbb Z$-modules 
\[\xymatrix{C_\bullet^\Delta: \mathbb Z Q_2\ar[r]^{\partial_1} & \mathbb Z Q_1\ar[r]^{\partial_0} &\mathbb Z Q_0 }\]
where $\partial_1(\tau)=\alpha_\tau+\beta_\tau+\gamma_\tau$ and $\partial_0(\alpha)=t(\alpha)-s(\alpha)$.  It satisfies $H_1(C_\bullet)\simeq H_1(S,\mathbb Z)$ by \cite[Lemma 2.2]{AG}.

Therefore if  $\gamma$ and $\gamma'$ are loops in $\pi_1^{\rm free}(S)$ with $[\gamma]=[\gamma']$ then 
$\overline{\gamma}^\Delta-\overline{\gamma'}^\Delta=\partial_1 (T)$ for some $T\in\mathbb Z Q_2$. So since $(d-d')(\tau)=d(\tau)-d'(\tau)=0$ for each $\tau\in Q_2$ and since $d$ and $d'$ are $\mathbb Z$-linear maps $\mathbb Z Q_1\to \mathbb Z$, we get 
\begin{align*}(d-d')(\gamma)-(d-d')(\gamma') & =d(\gamma)-d(\gamma')+d'(\gamma')-d'(\gamma)\\ & = d(\overline{\gamma}^\Delta)-d(\overline{\gamma'}^\Delta)+d'(\overline{\gamma'}^\Delta)-d'(\overline{\gamma}^\Delta\\  & = d(\overline{\gamma}^\Delta-\overline{\gamma'}^\Delta)-d'(\overline{\gamma}^\Delta-\overline{\gamma'}^\Delta) \\ &=d(\partial_1(T))-d'(\partial_1(T))=0.\end{align*}
Thus $d-d'$ is well-defined on the homology classes. The $\mathbb Z$-linearity folows directly from the $\mathbb Z$-linearity of $(d-d'):\mathbb Z Q_1\to \mathbb Z$. 
\end{proof}

\begin{corollary}\label{corollary d-d' homology}
Let $(\Delta,d)$ and $(\Delta',d')$ be graded triangulations. Then the map $d-d':\pi_1^{\rm free}(S)\to \mathbb Z$ factors through $[-]$ in a $\mathbb Z$-linear map $[d-d']:H_1(S,\mathbb Z)\to \mathbb Z$.
\end{corollary}

\begin{proof}
There let $s=i_1,\ldots, i_\ell$ be a sequence of flips such that $\mu_s(\Delta')=\Delta$. Then we have $\mu_s^L(\Delta',d')=(\Delta,d'')$ for some degree-$1$ map $d''$, and where $\mu_i^L$ is the left graded flip defined in \cite[Definition 2.13]{AG}. By the proof of \cite[Lemma 2.14]{AG} we get that $d'(\gamma)=d''(\gamma)$ for any closed curve $\gamma$ in $S$, that is $d'=d''$ as maps $\pi_1^{\rm free}(S)\to \mathbb Z$. 
Since $d$ and $d''$ are degree-$1$ maps $\mathbb Z Q_1^{\Delta}\to \mathbb Z$, we conclude by Lemma~\ref{lemma d-d' homology}.

\subsection{Derived invariant for general surface algebras}

Let us start with a piece of notation.
For an oriented surface $X$, with $Y,Z$ closed subsets of  $X$, we denote by $\Homeo^{+,Y}(X,Z)$ the group of orientation preserving homeomorphisms of $X$ whose restriction to $Y$ is the identity and which fix globally $Z$. The corresponding mapping class group, that is the group of elements in $\Homeo^{+,Y}(X,Z)$ up to isotopy among homeomorphisms of the same kind, is denoted by $\MCG^{Y}(X,Z)$. When $Y$ (resp. $Z$) is empty, we denote $\Homeo^{+}(X,Z)$ (resp. $\Homeo^{+,Y}(X)$) and $\MCG (X,Z)$ (resp. $\MCG ^{Y}(X)$) the corresponding groups.

The following result  is the main theorem in \cite{AG}.

\begin{theorem}\label{thmAG}\cite[Thm 3.13]{AG}
Let $(S,M)$ be an oriented surface with marked points on the boundary, and let $(\gamma_1,\ldots,\gamma_s)$ be a set of simple closed curves generating $H_1(S,\mathbb Z)$. Let $(\Delta,d)$ and $(\Delta',d')$ be triangulations with admissible cuts and $\Lambda=(\Delta,d)$ and $\Lambda'=(\Delta',d')$ be the corresponding surface algebras. Then the following are equivalent
\begin{enumerate}
\item $\Lambda$ and $\Lambda'$ are derived equivalent;
\item There exists $\Phi\in \Homeo^+(S,M)$ such that for any $i=1,\ldots,s$ $d(\gamma_i)=d'(\Phi(\gamma_i)).$
\end{enumerate}
\end{theorem}

Note that if $\Phi\in \Homeo^+(S,M)$ is isotopic to identity, then $\Phi(\gamma)$ is isotopic to $\gamma$ for any simple closed curve $\gamma$, so $(2)$ is equivalent to  
\begin{itemize}
\item[(2')]
\textit{There exists }$\Phi\in \MCG(S,M)$\textit{ such that for any }$i=1,\ldots,s$ $d(\gamma_i)=d'(\Phi(\gamma_i))$.
\end{itemize}

By Corollay \ref{corollary d-d' homology}, $(2)$ is also equivalent to 
\begin{itemize}
\item[(2'')] \textit{There exists }$\Phi\in \MCG(S,M)$ such that $[d-d'\circ\Phi]=0$ as a map $H_1(S,\mathbb Z)\to \mathbb Z$.
\end{itemize}

\end{proof}

\section{Derived invariant in the case $g=b=1$}

The aim of this section is to reformulate statement $(2)$ of Theorem~\ref{thmAG} in the case of the torus with one boundary component in order to get the following result.

\begin{theorem}\label{thmgcd} Let $\Lambda=(\Delta,d)$ and $\Lambda'=(\Delta',d')$ be surface algebras associated to a surface of genus $1$ with one boundary component. Let $(a,b)$ be simple closed curves that generate $H_1(S,\mathbb Z)$. Then the following are equivalent
\begin{enumerate}
\item the algebras $\Lambda$ and $\Lambda'$ are derived equivalent;
\item $\gcd (d(a),d(b))=\gcd (d'(a),d'(b))$;
\end{enumerate}
where $\gcd (0,n)$ is defined to be $n$.
\end{theorem}

This result is proved in subsection \ref{subsection proof}. The idea of the proof is first to understand the action of $\MCG(S,M)$ on $H_1(S,\mathbb Z)$ (Proposition \ref{MCG}), and then to show that $d$ defines a $\mathbb Z$-linear map $[d]:H_1(S,\mathbb Z)\to \mathbb Z$. For both steps we need first to consider the case of the once punctured torus which is easier.

\subsection{The mapping class group: Case of the once punctured torus}\label{classical}

We are interested in the action of $\mathcal{M}(S,M)$ on $H_1(S,\mathbb Z)$. This action is well known for the once punctured torus $\Sigma_{1,1}$. Our case is slightly more complicated as we have boundary and  do not impose homeomorphism to be the identity on it. We start this section by recalling results on the classical case of the once punctured torus that are useful for our purpose.

Let $\Sigma_{1,1}$ be the once punctured torus. Its fundamental group $\pi_1(\Sigma_{1,1})$ is free of rank two. Any choice of generators $a$ and $b$ induces bijections $\pi_1(\Sigma_{1,1})\simeq \mathbb{F}_2$ and $H_1(\Sigma_{1,1},\mathbb Z)\simeq \mathbb Z^2$. The abelianization map $\pi_1(\Sigma_{1,1})\to H_1(\Sigma_{1,1},\mathbb Z)$ induces a group morphism 
$$f_2:{\rm Out}(\pi_1(\Sigma_{1,1}))\to {\rm Aut}(H_1(\Sigma,\mathbb Z))\simeq {\rm GL}(2,\mathbb Z)$$ where ${\rm Out}(\pi_1(\Sigma_{1,1}))={\rm Aut}(\pi_1(\Sigma_{1,1}))/{\rm Inn}(\pi_1(\Sigma_{1,1})$ is the quotient of automorphisms by inner automorphisms.

The mapping class group $\MCG(\Sigma_{1,1})=\MCG^{\partial \Sigma_{1,1}}(\Sigma_{1,1})$ acts naturally on the fundamental group by outer automorphisms. Hence we have a map $f_1:\MCG(\Sigma_{1,1})\to {\rm Out}(\pi_1(\Sigma_{1,1}))$. Consider the following composition 
$$\xymatrix{f: \MCG(\Sigma_{1,1})\ar[r]^{f_1} & {\rm Out}(\pi_1(\Sigma_{1,1}))\ar[r]^{f_2}& {\rm GL}(2,\mathbb Z).}$$ The image is clearly contained in ${\rm SL}(2,\mathbb Z)$ since the homeomorphisms we are considering preserve orientation. Denote by $\tau_a$ and $\tau_b$ the Dehn twists around $a$ and $b$.  One verifies that 
$$\xymatrix@R=0.1cm{\tau_a: a\ar@{|->}[r] & a & & \tau_b: a\ar@{|->}[r] & ab\\ \quad b\ar@{|->}[r] & ba & & \quad  b\ar@{|->}[r] & b. }$$
Therefore we obtain $$f(\tau_a)=\begin{pmatrix}1 & 1\\ 0 &1\end{pmatrix} \qquad \textrm{and }\qquad f(\tau_b)=\begin{pmatrix}1 & 0\\ 1 &1\end{pmatrix}.$$ These two matrices generate ${\rm SL}(2,\mathbb Z)$ hence the image of  $f$ is ${\rm SL}(2,\mathbb Z)$.  Now by a result due to Nielsen the map $f_2$ is injective (see \cite[p. 233]{FM}). Hence if we denote by ${\rm Out}^+(\pi_1(\Sigma_{1,1}))$ the preimage of ${\rm SL}(2,\mathbb Z)$ through $f_2$ we get the surjective composition 
\begin{equation}\label{morphism}\xymatrix{ \MCG(\Sigma_{1,1})\ar[r]^-{f_1}  &{\rm Out}^{+}(\pi_1(\Sigma_{1,1}))\ar[r]^-{f_2}_-{\sim} & {\rm SL}(2,\mathbb Z)}.\end{equation} Note that it turns out that $f_1$ is also an isomorphism (see \cite[p.57]{FM}).

\subsection{The mapping class group: Case of the torus with one boundary component}

Let glue a once punctured disc on the boundary of $S$, we obtain a once punctured torus $\Sigma_{1,1}$. Denote by $i:S\to \Sigma_{1,1}$ the inclusion. There is a deformation retract $r:\Sigma_{1,1}\to S$ such that $r\circ i={\rm id}_S$, hence there are  isomorphisms $\pi_1(S)\simeq \mathbb F_2$  and $\varphi:H_1(S,M)\to\mathbb Z^2$.

For $\epsilon$ a basis of $H_1(S,\mathbb Z)$ and $\gamma\in \pi_1^{\rm free}(S)$, we denote by $[\gamma]^\epsilon\in\mathbb Z^2$ the coordinates of $\varphi([\gamma])$ in the basis $\varphi(\epsilon)$.

\begin{proposition}\label{MCG}
If $g=b=1$ then the mapping class group $\MCG(S,M)$ is a central extension of ${\rm SL}(2,\mathbb Z)$ by $\mathbb Z$.

 Moreover any basis $\epsilon$ of $H_1(S,\mathbb Z)$ induces a projection $G^{\epsilon}:\MCG(S,M)\to {\rm SL}(2,\mathbb Z)$ which is compatible with the action of $\MCG(S,M)$ on $H_1(S)$. More precisely for any $\Phi\in {\rm Homeo}^+(S,M)$ and for any closed curve $\gamma$ on $S$, we have $[\Phi(\gamma)]^\epsilon=G^\epsilon(\Phi).[\gamma]^\epsilon$.
\end{proposition}

We start with the following lemma for which we give a proof for completness. The argument follows the proof of Proposition 2.4 in \cite{FM}.

\begin{lemma}\label{lemma MCG}
Let $A$ be an annulus with boundary components $B_1$ and $B_2$ and with a finite set $M$ of marked points on $B_2$. Then there is an isomorphism of groups $\MCG^{B_1}(A,M)\simeq \mathbb Z$.
\end{lemma}

\begin{proof}
Denote by $\pi:\widetilde A=\mathbb R\times [0,1]\to A$ the universal cover of $A$, and assume that $\pi^{-1}(M)=\mathbb Z\times\{1\}$. Let $\Phi\in {\rm Homeo}^{+,B_1}(A,M)$ and $\widetilde \Phi$ be the lifting of $\Phi$ on $\widetilde A$ that fixes $(0,0)$. The homeomorphism $\Phi$ fixes globally $B_2$ since $B_2$ is a boundary component and since $B_1$ is fixed. Hence $\widetilde{\Phi}$ induces a homeomorphism of $\mathbb R\times \{1\}$. Moreover this homeomorphism induces a bijection $\mathbb Z\times\{1\}\to \mathbb Z\times \{1\}$ since $\Phi$ fixes globally $M$. Thus $\widetilde{\Phi}_{|_{\mathbb R\times \{1\}}}$ is a translation by an integer. Therefore we obtain a morphism of groups $\MCG^{B_1}(A,M)\to \mathbb Z$. It is surjective since for any $n\in \mathbb Z$ the transformation given by the matrix $$\begin{pmatrix} 1
 & n\\ 0 & 1\end{pmatrix}$$ is a homeomorphism of $\widetilde A$ which is $\pi$-equivariant. It induces then a homeomorphism of $A$ that preserves globally $M$ and fixes pointwise $B_1$. Now let $\Phi$ such that $\widetilde \Phi$ fixes pointwise $\mathbb Z\times \{1\}$. This homeomorphism is then isotopic to a homeomorphism $\Phi'$ such that $\widetilde{\Phi'}$ is the identity on $\mathbb R\times \{1\}$. Then the injectivity can be proven as in \cite[Prop 2.4]{FM}.
 
The mapping class group of this annulus is then generated by an homeomorphism that we could call the $\frac{1}{p}$-Dehn twist, where $p=\sharp M$.

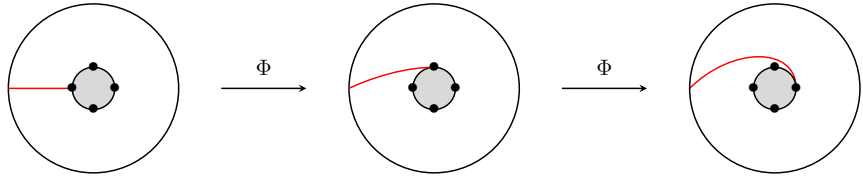
\[\scalebox{0.7}{
\begin{tikzpicture}[scale=0.8,>=stealth]
\draw [black, thick, fill=gray!30](0,0) circle (0.5);
\draw [black, thick](0,0) circle (2);
\draw[->,thick] (3,0)--(5,0);
\node at (4,0.5) {$\Phi$};
\draw [black, thick, fill=gray!30](8,0) circle (0.5);
\draw [black, thick](8,0) circle (2);
\draw[red, thick] (-2,0)-- (-0.5,0);
\draw[red, thick] (6,0).. controls (6,0) and (7,0.5).. (8,0.5);
\draw[red, thick] (14,0).. controls (15,1) and (16.5,1).. (16.5,0);
\draw [black, thick, fill=gray!30](16,0) circle (0.5);
\draw [black, thick](16,0) circle (2);
\draw[->,thick] (11,0)--(13,0);
\node at (12,0.5) {$\Phi$};

\node at (-0.5,0) {$\bullet$};
\node at (0,0.5) {$\bullet$};
\node at (0.5,0) {$\bullet$};
\node at (0,-0.5) {$\bullet$};
\node at (7.5,0) {$\bullet$};
\node at (8,0.5) {$\bullet$};
\node at (8.5,0) {$\bullet$};
\node at (8,-0.5) {$\bullet$};
\node at (15.5,0) {$\bullet$};
\node at (16,0.5) {$\bullet$};
\node at (16.5,0) {$\bullet$};
\node at (16,-0.5) {$\bullet$};

\end{tikzpicture}}\]

\end{proof}

\begin{proof}[Proof of Proposition \ref{MCG}.]

Fix $a$ and $b$ simple closed curves in minimal position on $S$ and generating $H_1(S,M)$, and fix an inclusion $i:S\to \Sigma_{1,1}$.  Let $\Phi$ be in $\Homeo^+(S,M)$, then $\Phi$ can be extended to a positive homeomorphism $\bar{\Phi}$ of the once punctured torus $\Sigma_{1,1}$.  Moreover $\bar{\Phi}$ is  uniquely defined up to homotopy since the homeomorphisms of the once punctured disc (not necessarily identity on the boundary) are isotopic to the identity. Moreover if $\Phi_1$ and $\Phi_2$ are equal in $\MCG(S,M)$ we cleraly have $\bar{\Phi}_1=\bar{\Phi}_2$ in $\MCG(\Sigma_{1,1})$. 

Now let $\Phi$ be a homeomorphism of $\Sigma_{1,1}$, and let $c$ be a curve surrounding the puncture. Then $\Phi(c)$ is isotopic to $c$, hence there exists $\Phi'$ isotopic to $\Phi$ such that $\Phi'$ fixes $c$ pointwise (by \cite[Proposition 1.11]{FM}). Hence the map $\MCG(S,M)\to \MCG(\Sigma_{1,1})$ is surjective.

Thus by \ref{morphism}, we have a surjective morphism 

$$\xymatrix{G:\MCG(S,M)\ar[r] & \MCG(\Sigma_{1,1}) \ar[r]& {\rm Out}^+(\mathbb F_2)\ar[r]^\sim_{\rm Nielsen} & {\rm SL}(2,\mathbb Z)}$$
which is compatible with the action of $\MCG(S,M)$ on $H_1(S,\mathbb Z)$ as $([i(a)],[i(b)])$ is a basis of $H_1(\Sigma_{1,1},\mathbb Z)$.

Now let $\Phi$ be a homeomorphism of $S$ fixing globally $M$ which is the kernel of $G$. Then by Nielsen theorem $\Phi$ acts trivially on $\pi_1(S)$, hence $\Phi(a)$ is isotopic to $a$ and $\Phi(b)$ is isotopic to $b$. Then up to isotopy we can assume that $\Phi$ fixes pointwise $a$ and $b$. If we cut $S$ along $a\cup b$ we obtain an annulus $A$ as in Lemma \ref{lemma MCG}, and $\Phi$ induces a homeomorphism of $A$ that fixes globally $M$ and pointwise the boundary component $a\cup b$. Hence by Lemma~\ref{lemma MCG}, we obtain the following short exact sequence of groups :
\[\xymatrix{(*)&  0\ar[r] & \mathbb Z\ar[r] & \MCG(S,M)\ar[r]^G & {\rm SL}(2,\mathbb Z)\ar[r] & 1}.\] 

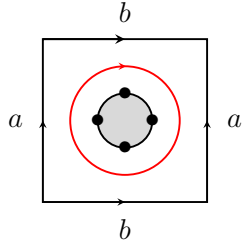
\[\scalebox{0.9}{
\begin{tikzpicture}[scale=0.8,>=stealth]

\draw[->] (0,0) to (0,1.5); \draw[ ->] (3,0) to (3,1.5);\draw[ ->] (0,0) to (1.5,0);\draw[thick, ->] (0,3) to (1.5,3);
\draw[thick] (0,0) rectangle (3,3);
\node at (-0.5,1.5) {$a$};
\node at (1.5,3.5) {$b$};
\node at (1.5,-0.5) {$b$};
\node at (3.5,1.5) {$a$};

\draw[thick, fill=gray!30] (1.5,1.5) circle (0.5);

\draw[thick, red] (1.5,1.5) circle (1);
\draw[->,red] (1.48,2.5)--(1.52,2.5);

\node at (1,1.5) {$\bullet$};
\node at (1.5,1) {$\bullet$};
\node at (1.5,2) {$\bullet$};
\node at (2,1.5) {$\bullet$};

\end{tikzpicture}}\]

Let us prove now that ${\rm Ker}\; G$ is in the center of $\MCG(S,M)$. Let $c$ be a curve surrounding $\partial S$, it cuts out $S$ into two connected components. Denote by $A$ the connected component containing $\partial S$, it is homeomorphic to an annulus with marked points on one boundary component.  Let $\Phi$ a homeomorphism in the kernel of $G$. Then up to isotopy we can assume that $\Phi$ is the identity on $S\backslash A$. Now let $\Phi'$ be a homeomorphism of $S$ fixing globally $M$, then up to isotopy we can assume that $\Phi'$ fixes pointwise $c$. Then we clearly have $\Phi\circ\Phi'=\Phi'\circ\Phi$ on $S\backslash A$. And we have $(\Phi'\circ\Phi)_{|_A}=(\Phi\circ \Phi')_{|_A}$ as elements in $\MCG^c(A,M)$ since the group $\mathbb Z$ is abelian. 

\end{proof}

\begin{remark}\label{remark Paris}
The exact sequence $(*)$ is not split and $\MCG(S,M)$ is not the direct product of ${\rm SL}(2,\mathbb Z)$ by $\mathbb Z$ as stated in a previous version. In fact $\MCG(S,M)$ is isomorphic to the Artin braid group $\langle A,B\, |\, ABA=BAB\rangle$ (see \cite{Chow}). We also refer to \cite[Section 5]{PR} for a precise and geometric desciption of the center of this group.
\end{remark}

\subsection{The grading factors through homology: Case of the once punctured torus}\label{subsection 2}

Let $\Sigma:=\Sigma_{1,1}$ be the once punctured torus, and $\Delta$ be an ideal triangulation of $\Sigma$.  Any ideal triangulation of $\Sigma$ has exactly two triangles $\tau$ and $\tau'$ and the associated quiver $Q^\Delta$ is the Markoff quiver. 

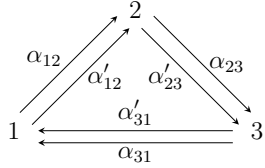
\[\scalebox{0.8}{
\begin{tikzpicture}[scale=1,>=stealth]
\node (A1) at (0,0) {$1$};
\node (A2) at (2,2) {$2$};
\node (A3) at (4,0) {$3$};
\draw[-> ] (0.1,0.3)-- node[xshift=-4mm,yshift=1mm]{$\alpha_{12}$}(1.7,1.9);
\draw[-> ] (0.3,0.1)-- node[xshift=4mm]{$\alpha'_{12}$}(1.9,1.7);
\draw[-> ] (2.1,1.7)-- node[xshift=-4mm]{$\alpha'_{23}$}(3.7,0.1);
\draw[-> ] (2.3,1.9)-- node[xshift=4mm]{$\alpha_{23}$}(3.9,0.3);
\draw[<- ] (0.4,0)-- node[yshift=3mm]{$\alpha'_{31}$}(3.6,0);
\draw[<- ] (0.4,-0.2)-- node[yshift=-2mm]{$\alpha_{31}$}(3.6,-0.2);

\end{tikzpicture}}\]

The aim of this subsection is to prove the following.

\begin{proposition}\label{propgenus1}
Let $\Sigma$ be the once punctured torus, and $\Delta$ be an ideal triangulation of $\Sigma$. Let $d$ be a degree-$1$ map of $Q^\Delta$. Let $\gamma$ be a simple closed non contractible curve on $\Sigma$. Then $d(\gamma)$ only depends on the homology class $[\gamma]$ of $\gamma$. 

Moreover, if $a$ and $b$ are simple closed curves such that $([a],[b])$ is a basis of $H_1(\Sigma,\mathbb Z)$ and if $[\gamma]=\lambda[a]+\mu[b]\neq 0$ then $d(\gamma)=\lambda d(a)+\mu d(b)$. 
\end{proposition}

The proof of this Proposition requires several lemmas (\ref{move33}, \ref{move42}, \ref{move51} and \ref{move60}) that are proved in the rest of the section.

\medskip

The once punctured torus $\Sigma$ with an ideal triangulation is presented  in Figure \ref{figure torus} by a square where parallel sides are identified (note that the puncture is at the corner of the square).  We fix a point $x_0$  in the triangle $\tau'$ containing the arrows $\alpha'_{12}$, $\alpha_{23}'$ and $\alpha_{31}'$.  We define $a$ and $b$ the generators of $\pi_1(\Sigma_{1,1},x_0)$ such that $\overline{a}^\Delta=\alpha_{23}-\alpha'_{23}$ and $\overline{b}^{\Delta}=\alpha_{12}-\alpha'_{12}$ (drawn in green in Figure \ref{figure torus}).

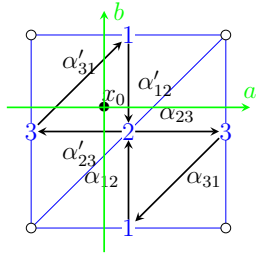
\begin{figure}\scalebox{0.8}{
\begin{tikzpicture}[scale=0.8,>=stealth]

\draw[blue] (0,0)--node (A1)[fill=white, inner sep=0pt]{$1$}(4,0)-- node (A3) [fill=white, inner sep=0pt]{$3$}(4,4)--node (A2)[fill=white, inner sep=0pt]{$2$}(0,0)--node (B3)[fill=white, inner sep=0pt]{$3$}(0,4)--node (B1)[fill=white, inner sep=0pt]{$1$}(4,4);
\node at (1.5,2.5) {$\bullet$};
\node at (1.7,2.7) {$x_0$};
\draw [->, thick] (A1)--node [left] {$\alpha_{12}$}(A2);\draw [->, thick] (A2)--node [above] {$\alpha_{23}$}(A3); \draw [->, thick] (A3)--node [right]{$\alpha_{31}$}(A1);

\draw [->, thick] (B3)--node [above] {$\alpha'_{31}$}(B1);\draw [->, thick] (B1)--node [right] {$\alpha'_{12}$}(A2); \draw [->, thick] (A2)--node [below]{$\alpha'_{23}$}(B3);

\draw[thick, green, ->] (-0.5,2.5)--(4.5,2.5);
\node[green] at (4.5,2.8) {$a$};
\draw[thick, green, ->] (1.5,-0.5)--(1.5,4.5);
\node[green] at (1.8,4.5){$b$};
\draw[fill=white] (0,0) circle (0.1);\draw[fill=white] (0,4) circle (0.1);\draw[fill=white] (4,0) circle (0.1);\draw[fill=white] (4,4) circle (0.1);

\end{tikzpicture}}
\caption{The once punctured torus $\Sigma$ with a triangulation and generators of $\pi_1(\Sigma,x_0)$.}\label{figure torus}
\end{figure}

Let $P:\mathbb R^2\backslash \mathbb Z^2\to \Sigma$ be the natural projection. 
For any loop $\gamma$ on $\Sigma$ starting at $x_0$, there exists $\widetilde{\gamma}:[0,1]\to \mathbb R^2\backslash \mathbb Z^2$ such that $P\circ \widetilde{\gamma}=\gamma$. Then $\widetilde{\gamma}(1)-\widetilde{\gamma}(0)$ is a vector in $\mathbb Z^2$ which is the homology class of $\gamma$ in the basis $([a],[b])$.
Note that if $\gamma$ is simple, so is $\widetilde{\gamma}$.

Let $G$ be the subgroup of words of even length of $\langle i_1,i_2,i_3, \ i_1^2=i_2^2=i_3^2=1\rangle$ . One can associate to $\gamma$ an element $w$ in $G$, determined by the successive arcs of the triangulation crossed by $\gamma$. This map induces an isomorphism of groups $\pi_1(\Sigma,x_0)\to G$. We define $G^{\rm cyc}$ as the set of conjugacy classes of $G$, that is $G^{\rm cyc}$ is the quotient of $G$ by the equivalence relation $w_1w_2\simeq w_2w_1$ where $w_1$ and $w_2$ have both even length. Note that $G^{\rm cyc}$ is in natural bijection with $\pi_1^{\rm free}(\Sigma)$, and so is not a group.

We define map $A:G\to \mathbb{Z}Q_1$ as follows:

$$ A(i_{\ell_1}\ldots i_{\ell_{2p}})= \sum_{k=1}^{p} \alpha_{\ell_{2k-1}\ell_{2k}} +\sum_{k=1}^{p-1}\alpha'_{\ell_{2k}\ell_{2k+1}}$$ for any $w=i_1\ldots i_{2p}$ reduced word in $G$, and where we use the convention $\alpha_{k\ell}=-\alpha_{\ell k}$. We define a map $A^{\rm cyc}:G^{\rm cyc}\to \mathbb Z Q_1$ by $ A^{\rm cyc}(i_{\ell_1}\ldots i_{\ell_{2p}})=A(i_{\ell_1}\ldots i_{\ell_{2p}})+\alpha'_{\ell_{2p}\ell_1},$ if $w=i_1\ldots i_{2p}$ is a reduced word in $G^{\rm cyc}$.

If $w\in G$ is the element corresponding to the loop $\gamma$ then we have $\overline{\gamma}^\Delta=A^{\rm cyc}(w)$.

\begin{lemma}\label{move33} Let $\gamma\in \pi_1(\Sigma,x_0)$ corresponding to a reduced word in $G^{\rm cyc}$ of the shape $w=w_1 i_1 i_2 i_1 i_3 i_1 w_2$. Let $\gamma'\in \pi_1(\Sigma,x_0)$  corresponding to the word $w'=w_1 i_1i_3i_1i_2i_1w_2$, then we have $d(\gamma)=d(\gamma')$. 
\end{lemma}

\begin{proof}
Assume that $w_1$ has even length, then we may assume that $w_1=\emptyset$ since $A^{\rm cyc}$ is well defined on elements of $G^{\rm cyc}$.
Let $\ell=2,3$ such that $w_2=w_2'i_\ell$. Then the curves $\widetilde{\gamma}$ (in red) and $\widetilde{\gamma'}$ (in blue) differ locally as in the following picture.
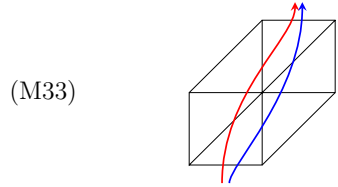
\[\scalebox{0.8}{\begin{tikzpicture}[scale=0.6,>=stealth]
\node at (-4,2) {(M33)};
\draw (0,0)--(2,0)--(2,4)--(4,4)--(4,2)--(0,2)--(0,0)--(4,4);
\draw (2,0)--(4,2);\draw (0,2)--(2,4);
\draw [thick, red,->](0.9,-0.5)..controls (0.9,2) and (2.9,3.5).. (2.9,4.5);
\draw [thick, blue,->](1.1,-0.5)..controls (1.1,0) and (3.1,2).. (3.1,4.5);
\end{tikzpicture}}\] 
Then we have the following equalities
\begin{align*}
A^{\rm cyc}(i_1 i_2 i_1 i_3 i_1 w_2)& = A(i_1i_2i_1i_3)+ \alpha'_{31}+A(w_2)+\alpha'_{\ell 1}\\
 &=  \alpha_{12} -\alpha'_{12} -\alpha_{31} +\alpha'_{31} +A(w_2) +\alpha'_{\ell 1}\\ 
  &=  A(i_1i_3i_1i_2)-\alpha'_{12}+A(w_2)+\alpha'_{\ell 1}\\
  &= A^{\rm cyc}(i_1i_3i_1i_2i_1w_2).
\end{align*}
The proof is similar for the case where $w_1$ has odd length and $w_2$ even length.
\end{proof}

\begin{lemma}\label{move42} Let $\gamma\in \pi_1(\Sigma,x_0)$ corresponding to a reduced word in $G^{\rm cyc}$ of the shape $w=w_1i_1i_2i_1i_3i_2i_3w_2$. Let $\gamma'$ be a loop corresponding to the word $w'=w_1i_1i_3i_1i_3w_2$, then we have $d(\gamma)=d(\gamma')$.\end{lemma}

\begin{proof}
Assume that $w_1=\emptyset$ and $w_2=i_mw'_2i_n$ has even length. Then $\widetilde{\gamma}$ and $\widetilde{\gamma'}$ are in the following configuration. 
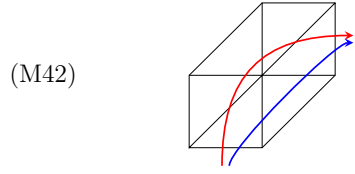
\[\scalebox{0.8}{
\begin{tikzpicture}[scale=0.6,>=stealth]
\draw (0,0)--(2,0)--(2,4)--(4,4)--(4,2)--(0,2)--(0,0)--(4,4);
\node at (-4,2) {(M42)};
\draw (2,0)--(4,2);\draw (0,2)--(2,4);

\draw [thick, red,->](0.9,-0.5)..controls (0.9,2) and (2,3.1).. (4.5,3.1);
\draw [thick, blue,->](1.1,-0.5)..controls (1.1,0) and (4,2.9).. (4.5,2.9);
\end{tikzpicture}}\]
We have the following equalities:
\begin{align*}
A^{\rm cyc}(w) &= \alpha_{12} -\alpha'_{12} -\alpha_{31}-\alpha'_{23}+\alpha_{23} +\alpha'_{3m}+A(w_2)+\alpha'_{n1}\\ & =\alpha_{31}+\alpha'_{31}-\alpha_{31}+\alpha'_{3m}+A(w_2)+\alpha'_{n1}+ \\ & \qquad +(\alpha_{31}+\alpha_{12}+\alpha_{23})-(\alpha'_{31}+\alpha'_{12}+\alpha'_{23})\\ &= A^{\rm cyc}(w')+  (\alpha_{31}+\alpha_{12}+\alpha_{23})-(\alpha'_{31}+\alpha'_{12}+\alpha'_{23}).
\end{align*}

Therefore one gets $d(\gamma)=d(\gamma')+d( (\alpha_{31}+\alpha_{12}+\alpha_{23})-(\alpha'_{31}+\alpha'_{12}+\alpha'_{23}))=d(\gamma')$ since $d( \alpha_{31}+\alpha_{12}+\alpha_{23})=1=d(\alpha'_{31}+\alpha'_{12}+\alpha'_{23})$.

The proof is similar if $w_1$ and $w_2$ have both odd length.

\end{proof}

\begin{lemma}\label{move51} Let $\gamma\in \pi_1(\Sigma,x_0)$ corresponding to a reduced word in $G^{\rm cyc}$ of the shape $w=w_1i_1i_2i_1i_3i_2i_1i_2w_2$. Let $\gamma'$ be a loop corresponding to the word $w'=w_1i_1i_3i_2w_2$, then we have $d(\gamma)=d(\gamma')$.\end{lemma}

\begin{proof}
If $w_1$ has even length, the curves $\widetilde{\gamma}$ and $\widetilde{\gamma'}$ are  in the following configuration.
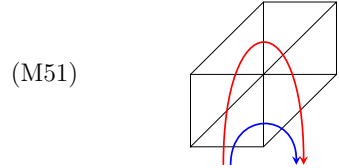
\[\scalebox{0.8}{
\begin{tikzpicture}[scale=0.6,>=stealth]
\draw (0,0)--(2,0)--(2,4)--(4,4)--(4,2)--(0,2)--(0,0)--(4,4);
\node at (-4,2) {(M51)};
\draw (2,0)--(4,2);\draw (0,2)--(2,4);

\draw [thick, red,->](0.9,-0.5)..controls (0.9,4) and (3.1,4).. (3.1,-0.5);
\draw [thick, blue,->](1.1,-0.5)..controls (1.1,1) and (2.9,1).. (2.9,-0.5);
\end{tikzpicture}}\]

It is enough to check that
\begin{align*}
A^{\rm cyc}(w)-A^{\rm cyc}(w') &= (\alpha_{12}-\alpha'_{12}-\alpha_{31}-\alpha'_{23}-\alpha_{12}+\alpha'_{12})-(-\alpha_{31}-\alpha'_{23})=0
\end{align*}
\end{proof}

\begin{lemma}\label{move60} Let $\gamma\in \pi_1(\Sigma,x_0)$ be a simple curve corresponding to a reduced word in $G^{\rm cyc}$ of containing the subword $i_1i_2i_1i_3i_2i_1i_3i_1$. Let $v$ be the word of maximal length such that $w$ has the following form $w=w_1vi_2i_1i_3i_2i_1i_3v^{-1}w_2$. Let $\gamma'$ be the curve corresponding to the word $w'=w_1w_2$, then we have $d(\gamma)=d(\gamma')$.\end{lemma}

\begin{proof}
Note that since $w$ is reduced, then $v$ is not empty and its last letter has to be $i_1$.
If $w_1v$ has even length, the curve $\widetilde{\gamma}$ looks as follows.

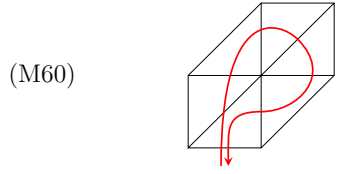
\[\scalebox{0.8}{
\begin{tikzpicture}[scale=0.6,>=stealth]
\draw (0,0)--(2,0)--(2,4)--(4,4)--(4,2)--(0,2)--(0,0)--(4,4);
\node at (-4,2) {(M60)};
\draw (2,0)--(4,2);\draw (0,2)--(2,4);
\draw [thick, red,->](0.9,-0.5)..controls (0.9,4) and (2.5,3.5).. (3,3)..controls (4,2) and (3,1).. (2,1).. controls (1,1) and (1.1,0.5).. (1.1,-0.5) ;
\end{tikzpicture}}\]

Since $\widetilde{\gamma}$ is a simple curve, then we are in one of the following configuration depending on the first letter of $v$ and on the parity of its length. 

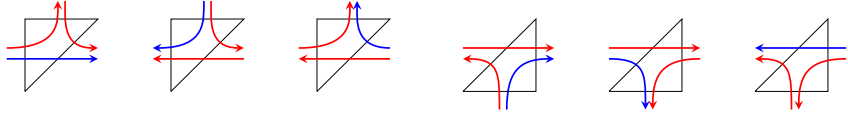
\[\scalebox{0.8}{
\begin{tikzpicture}[scale=0.6,>=stealth]
\draw (0,0)--(0,2)--(2,2)--(0,0);
\draw [thick, red,->] (-0.5,1.2)..controls (1,1.2) and (0.9,2)..(0.9,2.5);
\draw [thick, red,->] (1.1,2.5)..controls (1.1,2) and (1,1.2)..(2,1.2);
\draw[thick, blue, ->] (-0.5,0.9)--(2,0.9);
\draw (4,0)--(4,2)--(6,2)--(4,0);
\draw [thick, blue,<-] (3.5,1.2)..controls (5,1.2) and (4.9,2)..(4.9,2.5);
\draw [thick, red,->] (5.1,2.5)..controls (5.1,2) and (5,1.2)..(6,1.2);
\draw[thick, red, <-] (3.5,0.9)--(6,0.9);

\draw (8,0)--(8,2)--(10,2)--(8,0);
\draw [thick, red,->] (7.5,1.2)..controls (9,1.2) and (8.9,2)..(8.9,2.5);
\draw [thick, blue,<-] (9.1,2.5)..controls (9.1,2) and (9,1.2)..(10,1.2);
\draw[thick, red,<-] (7.5,0.9)--(10,0.9);

\draw (12,0)--(14,0)--(14,2)--(12,0);
\draw[thick,red,->] (12,1.2)--(14.5,1.2);
\draw[thick,red,<-] (12,0.9)..controls (13,0.9) and (13,0.5).. (13,-0.5);
\draw[thick,blue,->] (13.2,-0.5).. controls (13.2,0.5) and (13.5,0.9)..(14.5,0.9);

\draw (16,0)--(18,0)--(18,2)--(16,0);
\draw[thick,red,->] (16,1.2)--(18.5,1.2);
\draw[thick,blue,->] (16,0.9)..controls (17,0.9) and (17,0.5).. (17,-0.5);
\draw[thick,red,<-] (17.2,-0.5).. controls (17.2,0.5) and (17.5,0.9)..(18.5,0.9);

\draw (20,0)--(22,0)--(22,2)--(20,0);
\draw[thick,blue,<-] (20,1.2)--(22.5,1.2);
\draw[thick,red,<-] (20,0.9)..controls (21,0.9) and (21,0.5).. (21,-0.5);
\draw[thick,red,<-] (21.2,-0.5).. controls (21.2,0.5) and (21.5,0.9)..(22.5,0.9);
\end{tikzpicture}}\]

In the first configuration where $v$ has even length, $i_1$ as first letter and $i_k$ as last letter, we get  

\begin{align*}A^{\rm cyc}(w)-A^{\rm cyc}(w')= \alpha'_{31} +A(v)+\alpha'_{k1}+\alpha_{12}-\alpha'_{12}-\alpha_{31}-\alpha'_{23}-\alpha_{12}-\alpha'_{31}+\\ \qquad +\alpha_{31}
+\alpha'_{1k}+A(v^{-1})+\alpha'_{12}-(-\alpha'_{23})=0 \end{align*}

In the fourth configuration, $v$ has odd length, $i_2$ as first letter. Then $i_1$ is the last letter of $w_1$ and $i_3$ the first letter of $w_2$. Denote by $i_k$ the last letter of $v$.

\begin{align*} A^{\rm cyc}(w)-A^{\rm cyc}(w') & =  A(i_1v)+\alpha'_{k1}+ \alpha_{12}-\alpha'_{12}-\alpha_{31}-\alpha'_{23}-\alpha_{12}-\alpha'_{31}+\\  & \qquad +\alpha_{31}
+\alpha'_{1k}+A(v^{-1}i_3)-(- \alpha_{31}) \\  &= A(i_1v) -(\alpha'_{12}+\alpha'_{23}+\alpha'_{31})+A(v^{-1}i_3) +\alpha_{31} \\ & = \alpha_{12} +\alpha_{23} +\alpha_{31}  -(\alpha'_{12}+\alpha'_{23}+\alpha'_{31}). \end{align*}

Hence $d(\gamma)=d(\gamma')$.

The other cases can be checked similarly.

\end{proof}

\begin{proof}[Proof of Proposition \ref{propgenus1}]

Let $\gamma'$ be a simple closed curve on $\Sigma$ which is not contractible and with $[\gamma]=[\gamma']$. 
By the local moves $(M33)$, $(M42)$, $(M51)$ and $(M60)$ of lemmas \ref{move33}, \ref{move42}, \ref{move51} and \ref{move60} we can decrease the number of punctures which are in between the paths $\widetilde{\gamma}$ and $\widetilde{\gamma'}$, so that all intermediate steps are simple curves in $\mathbb{R}^2\backslash \mathbb{Z}^2$ having the same degree. Therefore $d(\gamma)$ is constant on the homology class of $\gamma$.

\medskip

Now assume that $[\gamma]\neq 0$, and denote by $a$ and $b$ the generators associated with the triangulation $\Delta$ (see Figure \ref{figure torus}). 

Let $\gamma'$ be the element $\gamma'=a^\lambda b^\mu$ (which corresponds to the reduced word $(i_2i_3)^\lambda(i_1i_2)^\mu=(i_2i_3)^{\lambda-1}(i_1i_2)^{\mu-1}i_1i_3$ in $G^{\rm cyc}$)  where $[\gamma]=\lambda[a]+\mu[b]$. 

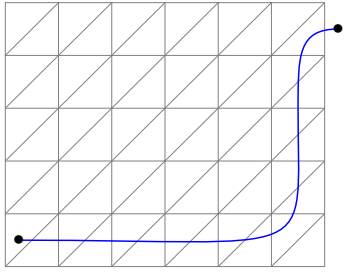
\[\scalebox{0.7}{
\begin{tikzpicture}[scale=1,>=stealth]

\draw[style=help lines, step=1] (0,0) grid (6,5);
\draw[style=help lines] (0,0)--(5,5);
\draw[style=help lines] (0,1)--(4,5);
\draw[style=help lines] (0,2)--(3,5);
\draw[style=help lines] (0,3)--(2,5);
\draw[style=help lines] (0,4)--(1,5);\draw[style=help lines] (1,0)--(6,5); \draw[style=help lines] (2,0)--(6,4); \draw[style=help lines] (3,0)--(6,3);
\draw[style=help lines] (4,0)--(6,2);\draw[style=help lines] (5,0)--(6,1);

\draw[thick, blue] (0.25,0.5)..controls (6,0.5) and (5.5,0).. (5.5,3);

\draw[thick, blue] (5.5,3)..controls (5.5,4) and (5.5,4.5).. (6.25,4.5);
\node at (0.25,0.5) {$\bullet$};
\node at (6.25,4.5) {$\bullet$};

\end{tikzpicture}}\]

Then an easy calcultation shows that \begin{align*}\overline{\gamma'}^\Delta& =(\lambda-2)(\alpha_{23}-\alpha'_{23}) +\alpha_{23}-\alpha'_{31}+(\mu-1)(\alpha_{12}-\alpha'_{12})-\alpha_{31}-\alpha'_{23}\\ & =\lambda (\alpha_{23}-\alpha'_{23})+\mu (\alpha_{12}-\alpha_{12}')+(\alpha_{12}'+\alpha'_{23}+\alpha'_{31})-(\alpha_{12}+\alpha_{23}+\alpha_{31})\\ & = \lambda \overline{a}^\Delta +\mu \overline{b}^\Delta +(\alpha_{12}'+\alpha'_{23}+\alpha'_{31})-(\alpha_{12}+\alpha_{23}+\alpha_{31}).\end{align*}
Since $d$ is a degree-$1$ map we get $d(\gamma')=\lambda d(a)+\mu d(b)$. Hence we get $d(\gamma)=\lambda d(a)+\mu d(b)$ by the first part of the proposition. So by putting $d([0]):=0$, $d$ induces a $\mathbb Z$-linear map $H_1(\Sigma,\mathbb Z)\to \mathbb Z$ which coincide with $d$ on the simple non contractible curves. The linearity then of course holds for any other basis of $H_1(\Sigma,\mathbb Z)$.

\end{proof}

\begin{remark}\label{remark} 
Proposition~\ref{propgenus1} is still true for any non contractible curve $\gamma$ in $\pi_1^{\rm free}(\Sigma)$ such that the corresponding curve $\widetilde{\gamma}$ in $\mathbb{R}^2\backslash\mathbb Z^2$ is simple.  Even if $\gamma$  and $\gamma'$ are simple as curves on $\Sigma$, when passing from $\gamma$ to $\gamma'$ by local moves, it is not clear that the intermediate curves are simple as curves on $\Sigma$, but they are simple as curves on  $\mathbb{R}^2\backslash\mathbb Z^2$. 

In the proof we use the fact that  $\widetilde{\gamma}$ is not trivial, which means that $\gamma$ is not homotopic to zero. Indeed the proposition is not true if $\gamma$ is contractible. For instance if $c$ is the curve surrounding the puncture, then its homology class vanishes, but we have $d(c)=2\neq 0=d(0)$. However for any non contractible closed curve $\gamma$ with $[\gamma]=0$ and which corresponds to a curve $\widetilde{\gamma}$ in $\mathbb R^2\backslash\mathbb Z^2$ without selfintersection, one can pass from $\widetilde{\gamma}$ to $\widetilde{c}$ by local moves $(M33)$, $(M42)$, $(M51)$ and $(M60)$, hence  we get $d(\gamma)=2$.

\end{remark}

\subsection{The grading factors through homology: Case of the torus with one boundary component}\label{subsection 3.4}

In this section $(S,M)$ is a surface of genus $1$ with one boundary component and with marked points on the boundary. 

\begin{proposition}\label{propgenus1.2} Let $(\Delta,d)$ be a graded triangulation on $(S,M)$. Let $\gamma$ be a simple closed curve on $S$ such that its class $[\gamma]$ in $H_1(S,\mathbb Z)$ is  not zero. Let $a,b$ be simple closed curves that generate $H_1(S,\mathbb Z)$, and assume $[\gamma]=\lambda [a]+\mu [b]$. Then we have $$d(\gamma)=\lambda d(a)+\mu d(b).$$
\end{proposition}

\begin{proof} 
Without loss of generality, we assume that $\gamma$ intersects transversally the arcs of $\Delta$ and with a minimal number of intersections.

We consider the following three kind of triangles (not necessarily internal) in $\Delta$: 
\begin{itemize}
\item the triangles \emph{homotopic to the boundary}; that are the triangles where the three sides are homotopic (or equal) to a boundary segment; 
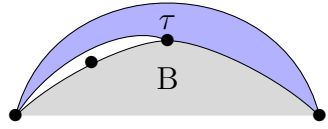
\[\scalebox{1}{
\begin{tikzpicture}[>=stealth,scale=1]

\draw[white,fill=gray!30] (0,0).. controls (0.5,0.5) and (1.5,1).. (2,1)..controls (2.5,1) and (3.5,0.5).. (4,0)--(0,0);

\draw (0,0).. controls (0.5,0.5) and (1.5,1).. (2,1);

\draw[fill=blue!30] (0,0).. controls (0.5,0.8) and (1.5, 1.25)..(2,1).. controls (2.5,1) and (3.5,0.5).. (4,0).. controls (3.5,2) and (0.5,2)..(0,0);

\node  at (0,0) {$\bullet$};
\node  at (2,1) {$\bullet$};
\node  at (4,0) {$\bullet$}; 
\node  at (1,0.7) {$\bullet$};

\node at (2,0.5) {B};

\node at (2,1.25) {$\tau$};

\end{tikzpicture}}\]

\item the triangles \emph{based on the boundary}; that are the triangles where exactly one side is homotopic (or equal) to a boundary segment;

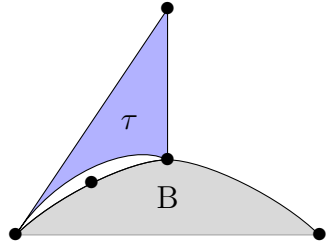
\[\scalebox{1}{
\begin{tikzpicture}[>=stealth,scale=1]

\draw[fill=gray!30] (0,0).. controls (0.5,0.5) and (1.5,1).. (2,1)..controls (2.5,1) and (3.5,0.5).. (4,0)--(0,0);

\draw[white] (0,0)--(4,0);

\draw (0,0).. controls (0.5,0.5) and (1.5,1).. (2,1);

\draw[fill=blue!30] (0,0).. controls (0.5,0.8) and (1.5, 1.25)..(2,1)--(2,3)--(0,0);

\node  at (0,0) {$\bullet$};
\node  at (2,1) {$\bullet$};
\node  at (4,0) {$\bullet$}; 
\node  at (1,0.7) {$\bullet$};
\node at (2,3) {$\bullet$};
\node at (2,0.5) {B};
\node at (1.5,1.5) {$\tau$};

\end{tikzpicture}}\]

 \item and the triangles which are \emph{uncontractible}, that are the triangles where the three side are not homotopic to a boundary segment.  
 \[\scalebox{1}
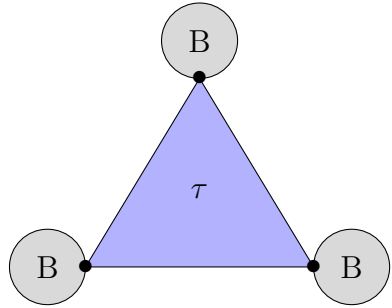
{
\begin{tikzpicture}[>=stealth,scale=1]

\draw[fill=gray!30] (0,0) circle (0.5);
\draw[fill=gray!30](4,0) circle (0.5);
\draw[fill=gray!30] (2,3) circle (0.5);

\draw[fill=blue!30] (0.5,0)--(3.5,0)--(2,2.5)--(0.5,0);

\node at (0.5,0) {$\bullet$};\node at (3.5,0) {$\bullet$};\node at (2,2.5) {$\bullet$};
\node at (0,0) {B}; \node at (4,0) {B}; \node at (2,3) {B}; 
\node at (2,1) {$\tau$};

\end{tikzpicture}}\]

 \end{itemize}
 The arrows coming from triangles homotopic to the boundary do not appear in $\overline{\gamma}^\Delta$. So we are just considering triangles of the other types. Then we can decompose \[\overline{\gamma}^\Delta=\sum_{\alpha\in Q_1}n_\alpha \alpha=\sum_{\underset{\tau\textrm{ based on the bdy}}{\alpha\in \tau }} n_\alpha \alpha +\sum_{\underset{ \tau\textrm{ uncontractible}}{\alpha\in \tau}}n_\alpha \alpha=\overline{\gamma}^\Delta_{\rm bb}+\overline{\gamma}^\Delta_{\rm uc}.\] 

\medskip

\textit{Step 1: We prove that $\overline{\gamma}^\Delta_{bb}$ depends only on the homology class of $\gamma$, and that $\overline{\gamma}^\Delta_{\rm bb}=\lambda \overline{a}^\Delta_{\rm bb}+\mu \overline{b}_{\rm bb}^\Delta$.}

Let $\tau$ be a triangle of $\Delta$ based on the boundary. Denote by $i$ and $j$ its two sides that are not homotopic to a boundary segment, and by $\alpha$ the corresponding arrow $i\to j$ in $Q^\Delta$. Since $\tau$ is based on the boundary, the arrow $\alpha$ is the only arrow of the triangle $\tau$ that may occur in $\overline{\gamma}^\Delta$. Moreover, for a well chosen orientation of the arcs $i$ and $j$, one easily checks that $n_\alpha={\rm Int}(i,\gamma)={\rm Int}(j,\gamma)$, where ${\rm Int}(i,\gamma)$ is the algebraic number of intersections between $i$ and $\gamma$ (see subsection \ref{subsection intersection}). Then by lemma \ref{lemma intersection} we immediately get Step 1. 
  
\medskip

\textit{Step 2: We prove that $d(\overline{\gamma}^\Delta_{\rm uc})$ depends only on the homology class of $\gamma$ and that $d(\overline{\gamma}^\Delta_{\rm uc})=\lambda d( \overline{a}^\Delta_{\rm uc})+\mu d(\overline{b}_{\rm uc}^\Delta)$. .}

The natural homeomorphsim $f:S\backslash \partial S \to \Sigma_{1,1}$ induces a bijection $\pi_1^{\rm free}(S)\simeq \pi_1^{\rm free}(\Sigma_{1,1})$ and a group isomorphism $H_1(S,\mathbb Z)\simeq H_1(\Sigma_{1,1},\mathbb Z)$. Hence we have $[f(\gamma)]=\lambda [f(a)]+\mu [f(b)]\neq 0$.

There are only two uncontractible triangles $\tau$ and $\tau'$ in $\Delta$. Moreover $f(\tau)$ and $f(\tau ')$ give a triangulation $\Delta'$ of $\Sigma_{1,1}$. Then $f$ induces a bijection between the arrows of $Q^\Delta$ coming from $\tau$ and $\tau'$ in $\Delta$ with the arrows of $Q^{\Delta'}$. Denote by $d'\in \mathbb Z^{Q^{\Delta'}_1}$ the grading induced by $d$, it is clearly a degree-$1$ map on $Q^{\Delta'}$.  We immediately get that $\overline{\gamma}^\Delta_{\rm uc}= \overline{f(\gamma)}^{\Delta'}$ through the bijection described above. Hence we get

\begin{align*}
d(\overline{\gamma}^\Delta_{\rm uc}) & = d'(\overline{f(\gamma)}^{\Delta'}) = \lambda d'(\overline{f(a)}^{\Delta'})+\mu d'(\overline{f(b)}^{\Delta'}) & \textrm{by Proposition \ref{propgenus1}}\\ & = \lambda d(\overline{a}^\Delta_{\rm uc})+\mu d(\overline{b}^\Delta_{\rm uc}).
\end{align*}

\medskip

Combining Steps 1 and 2, we immediately get the result.

\end{proof}

\subsection{Proof of Theorem \ref{thmgcd}}\label{subsection proof}

We start with the following classical result for which we give a proof for completeness.

\begin{lemma}\label{orbit} Let $(m,n)\in \mathbb Z^2$. Then there exists $M\in {\rm SL}(2,\mathbb Z)$ such that $\begin{pmatrix}m'\\n'\end{pmatrix}=M.\begin{pmatrix}m\\n\end{pmatrix}$ if and only if $\gcd(m,n)=\gcd(m',n')$.
\end{lemma}

\begin{proof}
Let $t\in \mathbb N$. Then $(m,n)\in\mathbb Z^2$ is in the orbit of $(t,0)$ if and only if $t$ divides $m$ and $n$, and there exists $u$ and $v$ such that $\frac{m}{t}u-\frac{n}{t}v=1$. By Bezout's identity, this is equivalent to the fact that $t$ divides $\gcd (m,n)$ and that $\gcd (\frac{m}{t},\frac{n}{t})=1$, so $\gcd(m,n)=t$.
\end{proof}

\begin{proof}[Proof of Theorem \ref{thmgcd}]$ $
\medskip

$(1)\Rightarrow (2)$. By Theorem \ref{thmAG}, there exists $\Phi\in\mathcal{M}(S,M)$ such that $d(a)=d'(\Phi (a))$ and $d(b)=d'(\Phi (b))$. Denote by $\epsilon=([a],[b])$ the basis of $H_1(S,\mathbb Z)$, and $M=G^\epsilon(\Phi)=\begin{pmatrix}p & q\\ r & s\end{pmatrix}$ the matrix of ${\rm SL}(2,\mathbb Z)$ corresponding to $\Phi$. Then by Proposition \ref{MCG} we have \[[\Phi(a)]^\epsilon=M.[a]^\epsilon=\begin{pmatrix}p\\ r\end{pmatrix}\quad \textrm{and} \quad[\Phi(b)]^\epsilon=M.[b]^\epsilon=\begin{pmatrix}q\\ s\end{pmatrix},\] that is  \[[\Phi(a)]=p[a]+r[b]\quad \textrm{and} \quad[\Phi(b)]=q[a]+s[b].\]
Since $[\Phi(a)]$ and $[\Phi(b)]$ are not zero we get \[\tag{*}\begin{pmatrix}d'(\Phi(a))\\ d'(\Phi(b))\end{pmatrix}  =\begin{pmatrix}pd'(a)+rd'(b)\\ qd'(a)+sd'(b)\end{pmatrix} = M. \begin{pmatrix}d'(a)\\ d'(b)\end{pmatrix}\] by Proposition~\ref{propgenus1.2}. By hypothesis $d(a)=d'(\Phi(a))$ and $d(b)=d'(\Phi(b))$, so we conclude using Lemma \ref{orbit}.

\medskip

$(2)\Rightarrow (1)$. By Lemma \ref{orbit}, if $\gcd(d(a),d(b))=\gcd (d'(a),d'(b))$ there exists $M\in{\rm SL}(2,\mathbb Z)$ such that $\begin{pmatrix}d(a)\\d(b)\end{pmatrix}=M.\begin{pmatrix}d'(a)\\d'(b)\end{pmatrix}$. Let $\Phi\in \Homeo^+(S,M)$ such that $G^\epsilon(\Phi)=M$. Then by the equalitiy (*), we get $d(a)=d'(\Phi(a))$ and $d(b)=d'(\Phi(b))$. We conclude using Theorem~\ref{thmAG}.

\end{proof}

\noindent
\textbf{Notation.}
For a surface algebra $\Lambda=(\Delta,d)$ associated to the torus with one boundary component, we denote by $\gcd(\Lambda)$ the derived invariant $\gcd(d(a),d(b))$. This invariant does not depend on the choice of the generators $a$ and $b$ by Theorem \ref{thmgcd}.

\subsection{Examples and consequences}

Consider the following triangultation of the torus with one boundary component and $2$ marked points ($g=b=1$ and $p=2$) and fix two generators $a$ and $b$ of $\pi_1(S)$.
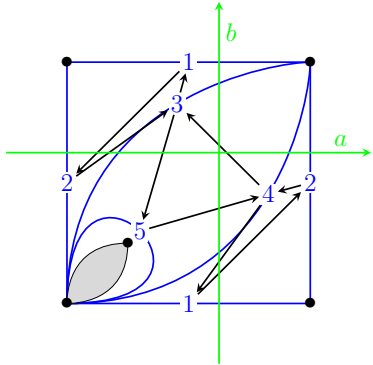
\[\scalebox{0.8}{
\begin{tikzpicture}[scale=1,>=stealth]
\draw[thick, blue] (0,0)-- node (A)[fill=white, inner sep=1pt]{$1$} (4,0);
\draw[thick, blue] (4,0)-- node (B1)[fill=white, inner sep=1pt]{$2$} (4,4);
\draw[thick, blue] (4,4)-- node (A1)[fill=white, inner sep=1pt]{$1$} (0,4);
\draw[thick, blue] (0,4)-- node (B)[fill=white, inner sep=1pt]{$2$} (0,0);

\draw[thick, blue] (0,0).. node (E)[fill=white, inner sep=1pt]{$5$} controls (3.2,0) and (0,3.2).. (0,0);
\draw[thick, blue] (0,0).. node (D)[fill=white, inner sep=1pt]{$4$} controls (3.5,0) and (4,3.5).. (4,4);
\draw[thick, blue] (0,0).. node (C)[fill=white, inner sep=1pt]{$3$} controls (0,3.5) and (3.5,4).. (4,4);

\draw[->, thick] (A)--(B1); \draw[->, thick] (B1)--(D); \draw[->, thick] (D)--(A); \draw[->, thick] (D)--(C); \draw[->, thick] (C)--(E); \draw[->, thick] (E)--(D); \draw[->, thick] (C)--(A1); \draw[->, thick] (A1)--(B); \draw[->, thick] (B)--(C);

\draw {(0,0).. controls (0.1,0) and (1,0).. (1,1).. controls  (0,1) and (0,0.1).. (0,0)}[fill=gray!30];
\draw[thick, green, ->] (-1,2.5)--(5,2.5);\draw[thick, green, ->] (2.5,-1)--(2.5,5);
\node[green] at (4.5,2.7) {$a$};
\node[green] at (2.7,4.5) {$b$}; 
\node at (0,0) {$\bullet$};
\node at (4,0) {$\bullet$};
\node at (0,4) {$\bullet$};
\node at (4,4) {$\bullet$};
\node at (1,1) {$\bullet$};

\end{tikzpicture}}\]

One easily verifies that $\overline{a}^{\Delta}=\alpha_{23}-\alpha_{43}-\alpha_{42}$ and $\overline{b}^{\Delta}=-\alpha_{41}+\alpha_{43}+\alpha_{31}$.

Now consider the following different admissible cuts of $\Delta$, and the corresponding algebras $\Lambda_0$, $\Lambda_1$ and $\Lambda_2$ of global dimension $2$. The angles of degree $1$ are marked in red, and the zero relations in the quiver are marked with dots.

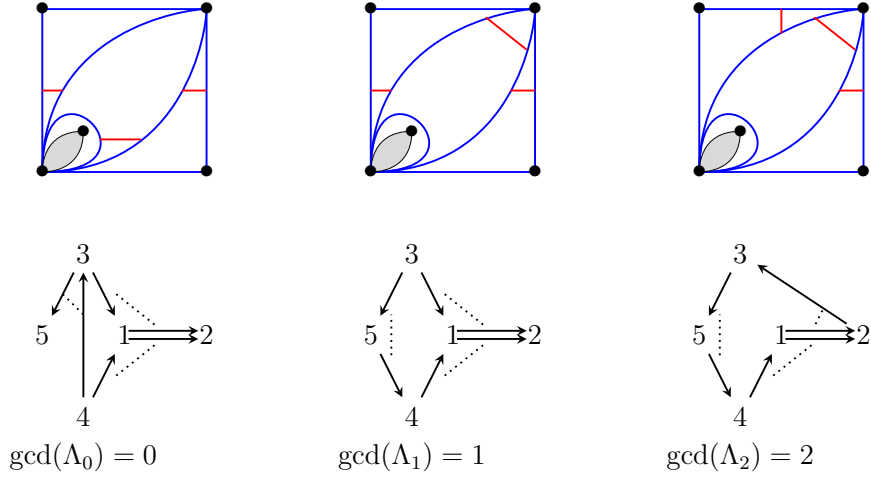
\[\scalebox{0.9}{
\begin{tikzpicture}[scale=0.6,>=stealth]
\draw[thick, blue] (0,0)-- (4,0);
\draw[thick, blue] (4,0)--  (4,4);
\draw[thick, blue] (4,4)--  (0,4);
\draw[thick, blue] (0,4)--  (0,0);

\draw[thick, blue] (0,0)..  controls (3.2,0) and (0,3.2).. (0,0);
\draw[thick, blue] (0,0).. controls (3.5,0) and (4,3.5).. (4,4);
\draw[thick, blue] (0,0).. controls (0,3.5) and (3.5,4).. (4,4);

\draw {(0,0).. controls (0.1,0) and (1,0).. (1,1).. controls  (0,1) and (0,0.1).. (0,0)}[fill=gray!30];
\node at (0,0) {$\bullet$};
\node at (4,0) {$\bullet$};
\node at (0,4) {$\bullet$};
\node at (4,4) {$\bullet$};
\node at (1,1) {$\bullet$};

\draw[thick, red] (0,2)--(0.5,2);\draw[thick, red] (1.4,0.8)--(2.4,0.8);\draw[thick, red] (3.4,2)--(4,2);

\node  (A5) at (0,-4) {$5$};
\node (A1) at (2,-4) {$1$};
\node (A2) at (4,-4) {$2$};
\node (A3) at (1,-2) {$3$};
\node (A4) at (1,-6) {$4$};
\draw[->, thick] (A3)--(A5); \draw[->,thick] (A4)--(A3); \draw[->, thick] (A3)--(A1); \draw[->,thick] (A4)--(A1);
\draw[->,thick] (2.1,-4.1)--(3.8,-4.1); \draw[->,thick] (2.1,-3.9)--(3.8,-3.9);
\draw[thick, dotted] (0.5,-3)--(1,-3.5); \draw[thick, dotted] (1.8,-3)--(2.8,-3.8);\draw[thick, dotted] (1.8,-5)--(2.8,-4.2);

\draw[thick, blue] (8,0)-- (12,0);
\draw[thick, blue] (12,0)--  (12,4);
\draw[thick, blue] (12,4)--  (8,4);
\draw[thick, blue] (8,4)--  (8,0);

\draw[thick, blue] (8,0)..  controls (11.2,0) and (8,3.2).. (8,0);
\draw[thick, blue] (8,0).. controls (11.5,0) and (12,3.5).. (12,4);
\draw[thick, blue] (8,0).. controls (8,3.5) and (11.5,4).. (12,4);

\draw {(8,0).. controls (8.1,0) and (9,0).. (9,1).. controls  (8,1) and (8,0.1).. (8,0)}[fill=gray!30];
\node at (8,0) {$\bullet$};
\node at (12,0) {$\bullet$};
\node at (8,4) {$\bullet$};
\node at (12,4) {$\bullet$};
\node at (9,1) {$\bullet$};

\draw[thick, red] (8,2)--(8.5,2);\draw[thick, red] (10.8,3.8)--(11.8,3);\draw[thick, red] (11.4,2)--(12,2);

\node  (B5) at (8,-4) {$5$};
\node (B1) at (10,-4) {$1$};
\node (B2) at (12,-4) {$2$};
\node (B3) at (9,-2) {$3$};
\node (B4) at (9,-6) {$4$};
\draw[->, thick] (B3)--(B5); \draw[->,thick] (B5)--(B4); \draw[->, thick] (B3)--(B1); \draw[->,thick] (B4)--(B1);
\draw[->,thick] (10.1,-4.1)--(11.8,-4.1); \draw[->,thick] (10.1,-3.9)--(11.8,-3.9);
 \draw[thick, dotted] (9.8,-3)--(10.8,-3.8);\draw[thick, dotted] (9.8,-5)--(10.8,-4.2);
 \draw[thick, dotted] (8.5,-4.5)--(8.5,-3.5);

\draw[thick, blue] (16,0)-- (20,0);
\draw[thick, blue] (20,0)--  (20,4);
\draw[thick, blue] (20,4)--  (16,4);
\draw[thick, blue] (16,4)--  (16,0);

\draw[thick, blue] (16,0)..  controls (19.2,0) and (16,3.2).. (16,0);
\draw[thick, blue] (16,0).. controls (19.5,0) and (20,3.5).. (20,4);
\draw[thick, blue] (16,0).. controls (16,3.5) and (19.5,4).. (20,4);

\draw {(16,0).. controls (16.1,0) and (17,0).. (17,1).. controls  (16,1) and (16,0.1).. (16,0)}[fill=gray!30];

\node at (16,0) {$\bullet$};
\node at (20,0) {$\bullet$};
\node at (16,4) {$\bullet$};
\node at (20,4) {$\bullet$};
\node at (17,1) {$\bullet$};

\draw[thick, red] (18,3.4)--(18,4);\draw[thick, red] (18.8,3.8)--(19.8,3);\draw[thick, red] (19.4,2)--(20,2);

\node  (C5) at (16,-4) {$5$};
\node (C1) at (18,-4) {$1$};
\node (C2) at (20,-4) {$2$};
\node (C3) at (17,-2) {$3$};
\node (C4) at (17,-6) {$4$};
\draw[->, thick] (C3)--(C5); \draw[->,thick] (C5)--(C4); \draw[->, thick] (C2)--(C3); \draw[->,thick] (C4)--(C1);
\draw[->,thick] (18.1,-4.1)--(19.8,-4.1); \draw[->,thick] (18.1,-3.9)--(19.8,-3.9);
 \draw[thick, dotted] (19,-3.4)--(18.8,-3.8);\draw[thick, dotted] (17.8,-5)--(18.8,-4.2);
 \draw[thick, dotted] (16.5,-4.5)--(16.5,-3.5);
 
 \node at (1,-7) {$\gcd(\Lambda_0)=0$};
 \node at (9,-7) {$\gcd(\Lambda_1)=1$};
 \node at (17,-7) {$\gcd(\Lambda_2)=2$};

\end{tikzpicture}}\]

Then we verify that $d_0(a)=1-0-1=0$ and $d_0(b)=-0+0+0=0$ so $\gcd(\Lambda_0)=0$. Similarly we obtain $\gcd(\Lambda_1)=\gcd(-1,1)=1$ and $\gcd(\Lambda_2)=\gcd(-2,2)=2$. Therefore these algebras are not derived equivalent. However they have the same AG-invariant in the sense of \cite{AAG}. Indeed, let $c$ be the loop around the boundary component, then the AG invariant is given by $(d(c)+2,2d(c)+2)$ by \cite[Prop. 5.3]{AG}, and in all these cases we have $d(c)=2$ (see \cite[Prop. 2.7]{AG}). So we observe the following.

\begin{corollary}\label{cor AG} There exist non derived equivalent gentle algebras with the same AG-invariant.
\end{corollary}

We can also prove the following result.

\begin{proposition}\label{prop values} Let $\Lambda$ be a surface algebra associated with a graded triangulation of a torus with one boundary component. Then $0\leq \gcd(\Lambda)\leq \frac{p+2}{2}$ where $p$ is the number of marked points. \end{proposition}

Note that we prove in the next section (Corollary \ref{cor values}) that all integral values in the interval $[0,\frac{p+2}{2}]$ occur.

\begin{proof} Let $\Lambda=(\Delta,d)$ ba a surface algebra associated with a graded triangulation of the torus with one boundary component.
Let $f:S\backslash\partial S\to \Sigma_{1,1}$ be a homeomorphism from the surface $S$ to the one-punctured torus $\Sigma_{1,1}$. Recall from Subsection \ref{subsection 3.4} that there are three kind of triangles in the triangulation $\Delta$:
\begin{itemize}
\item the uncontractible triangles, which are the only two triangles $\tau$ of $\Delta$ such that $f(\tau)$ is a triangle of $f(\Delta)$;
\item the triangles based on the boundary, which are the triangles $\tau$ such that $f(\tau)$ is (homotopic to) an arc of $f(\Delta)$;
\item the triangles homotopic to the bounday, which are the triangles that are contractible to the puncture via $f$.
\end{itemize}
We denote by $1$, $2$ and $3$ the arcs of $f(\Delta)$, and by $\alpha_{12}$, $\alpha_{23}$, $\alpha_{31}$, $\alpha'_{12}$, $\alpha'_{23}$, $\alpha'_{31}$ the arrows of $Q^\Delta$ which are in the uncontractible triangles as in Figure \ref{figure torus}. Denote by $p_1$ (resp. $p_2$, $p_3$) the number of triangles $\tau$ of $\Delta$ which are internal, based on the boundary, and such that $f(\tau)$ is homotopic to the arc $1$ (resp. $2$, $3$). And let $q$ be the number of internal triangles homotopic to the boundary. Let $a$ and $b$ be simple curves that generate $\pi_1(S)$ and such that $f(a)$ and $f(b)$ are in the same configuration as in Figure \ref{figure torus}. Then we have\[ d(\overline{a}^\Delta)=d(\overline{a}_{\rm uc})+d(\overline{a}_{\rm bb})=-d(\alpha'_{23})+d(\alpha_{23})+d(\overline{a}_{\rm bb}).\]
Since the loop $f(a)$ intersects exaclty once the arc $2$ and the arc $3$,  $\overline{a}_{\rm bb}\in \mathbb Z Q_1$ is a linear combination, with coefficient $\pm 1$, of all the arrows coming from triangles $\tau$ based on the boundary such that $f(\tau)$ is the arc $2$ or $3$. Moreover since $d$ is an admissible cut, among these triangles, only the ones that are internal contain an arrow of degree $1$, so we have $|d(\overline{a}_{\rm bb})|\leq p_2+p_3$ and thus $|d(\overline{a}^\Delta)|\leq p_2+p_3+1$.

For the same reasons we get $|d(\overline{b}^\Delta)|\leq p_1+p_3+1$. Hence we obtain the following inequality
\[ \gcd(\Lambda)\leq {\rm max\ }(|d(\overline{a}^\Delta)|,|d(\overline{b}^\Delta)|)\leq p_1+p_2+p_3+1.\]

Now the number of internal triangles is $\sharp Q_2=2+(p_1+p_2+p_3)+q$. Since there are three arrows for each internal triangle and one arrow for each triangle with one side on the boundary we get $\sharp Q_1\geq 3(2 +p_1+p_2+p_3+q)$. The number of vertices of the quiver is $\sharp Q_0=6g-6+3b+p=p+3$. From the equality $\sharp Q_0+\sharp Q_2-\sharp Q_1+1=0$ (cf \ref{equationQ}) we obtain 
\[2(p_1+p_2+p_3)\leq p-2q\leq p.\] Therefore we obtain $\gcd (\Lambda)\leq \frac{p+2}{2}$.

\end{proof}

\section{Application to gentle two-cycle algebras}

In this section we relate our result to \cite{BM} on gentle two cycles algebras. In particular Theorem \ref{thmgcd} gives a partial positive answer to the conjecture stated in the introduction of \cite{BM}.  

First recall that a gentle $\ell$-cycle algebra is a gentle algebra $\Lambda=kQ/I$ with the property that $\sharp Q_1-\sharp Q_0 =\ell -1$. 

Let $\Delta$ be an ideal triangulation of a surface $(S,M)$ of genus $g$ with $b$ boundary components. 
\begin{lemma} Let $C_\bullet$ be the complex of $\mathbb Z$-modules \[\xymatrix{C_\bullet:\mathbb Z Q_2\ar[r]^-{\partial_1} & \mathbb Z Q_1\ar[r]^{\partial_0} & \mathbb Z Q_2}\] defined by $\partial_1(\tau)=\alpha_\tau+\beta_\tau+\gamma_\tau$ and $\partial_0(\alpha)=t(\alpha)-s(\alpha)$. Its homology satisfies $H_0(C_\bullet)\simeq \mathbb Z$, $H^1(C_\bullet)\simeq \mathbb Z^{2g+b-1}$ and $H_2(C_\bullet)=0$. 
\end{lemma}

\begin{proof}
Any arrow $\alpha$ belongs to at most one internal triangle in $Q_2$, so $\partial_1$ is injective. Furthermore the quiver is connected so the kernel of the map $\mathbb Z Q_0\to \mathbb Z$ associating $1$ to any vertex $i\in Q_0$ is generated by $\{ t(\alpha)-s(\alpha),\  \alpha\in Q_1\}$ so the cokernel of $\partial_0$ is isomorphic to $\mathbb Z$. The fact that $H_1(C_\bullet)\simeq\mathbb Z^{2g+b-1}$ is proved in \cite[Lemma 2.3]{AG}.
\end{proof}

Hence we have \begin{equation}\label{equationQ}\sharp Q_1^\Delta -\sharp Q_0^\Delta -\sharp Q_2^\Delta = 2g+b-2.\end{equation}

Now let $d:Q_1^\Delta\to Z$ be an admissible cut and $\Lambda =(\Delta,d)$ be the corresponding surface algebra. Then we have $\sharp Q_0^\Lambda=\sharp Q_0^\Delta$ and $\sharp Q_1^\Lambda=\sharp Q_1^\Delta-\sharp Q_2^\Delta$. So together with the results of \cite{AO2} we deduce the following.

\begin{corollary}
The surface algebra $\Lambda=(\Delta,d)$ is a gentle 
\begin{itemize}
\item $0$-cycle algebra if and only if $g=0$ and $b=1$, that is if and only if $\Lambda$ is an algebra of global dimension $2$ of derived type $A_n$;
\item $1$-cycle algebra if and only if $g=0$ and $b=2$, that is if and only if $\Lambda$ is an algebra of global dimension $2$ of cluster type $\widetilde{A}_{p,q}$ (in the sense of \cite{AO2});
\item $2$-cycle algebra if and only if $g=0$ and $b=3$ or $g=b=1$.
\end{itemize}
\end{corollary}

In the case $g=0$ and $b=3$, it is proved in \cite{DRS} that $\sum_{(m,n)}\Phi_\Lambda(m,n)=3$ where $\Phi_\Lambda$ is the invariant introduced by Avella-Alaminos and Geiss, that is $\Lambda$ is nondegenerate in the sense of \cite{BM}. It follows also from \cite[Corollary 5.6]{AG} that the derived equivalence class of $\Lambda$ is exactly determined by its AG-invariant. This was also proved in \cite[Theorem~1]{BM} and in \cite{AA}. 

In the case $g=b=1$ then we have  $\sum_{(m,n)}\Phi_\Lambda(m,n)=1$, so $\Lambda$ is degenerate in the sense of \cite{BM}. For $s\in \mathbb N^*$ and $r=0,\ldots, s-1$, Bobinski and Malicki defined the algebra $\Lambda_0(s,r)$ by the quiver:

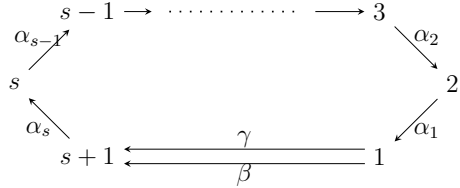
\[\scalebox{0.8}{
\begin{tikzpicture}[scale=1.2,>=stealth]

\node (A) at (0,0) {$s+1$}; 
\node (B) at (-1,1) {$s$}; 
\node (C) at (0,2) {$s-1$}; \node (D) at (1,2){}; \node (E) at (3,2){}; \node (F) at (4,2) {$3$}; \node (G) at (5,1) {$2$}; \node (H) at (4,0) {$1$}; 

\draw[->] (A)--node [xshift=-5pt, yshift=-5pt]{$\alpha_s$}(B);\draw[->] (B)--node [xshift=-5pt, yshift=5pt]{$\alpha_{s-1}$}(C);\draw[->](C)--(D);\draw[loosely dotted, thick](D)--(E);\draw[->] (E)--(F);\draw[->] (F)--node [xshift=5pt, yshift=5pt]{$\alpha_2$}(G); \draw[->] (G)--node [xshift=5pt, yshift=-5pt]{$\alpha_1$}(H);
\draw [->](3.8,0.1)--node [yshift=5pt]{$\gamma$} (0.5,0.1);\draw [->](3.8,-0.1)--node [yshift=-5pt]{$\beta$} (0.5,-0.1);
\end{tikzpicture}}\]
with the relations $\alpha_s\beta$, $\gamma\alpha_1$ and $\alpha_i\alpha_{i+1}$ for $i=1,\ldots, r$.

For $s\in \mathbb N^*$, they define an algebra $\Lambda_0'(s,0)$ by the quiver 

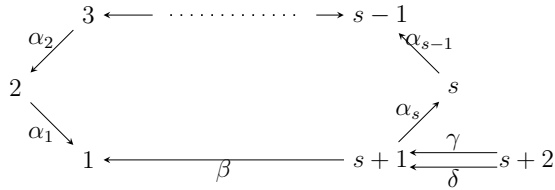
\[\scalebox{0.8}{
\begin{tikzpicture}[scale=1.2,>=stealth]

\node (A) at (0,0) {$1$}; 
\node (B) at (-1,1) {$2$}; 
\node (C) at (0,2) {$3$}; \node (D) at (1,2){}; \node (E) at (3,2){}; \node (F) at (4,2) {$s-1$}; \node (G) at (5,1) {$s$}; \node (H) at (4,0) {$s+1$}; 
\node (I) at (6,0) {$s+2$};

\draw[<-] (A)--node [xshift=-5pt, yshift=-5pt]{$\alpha_1$}(B);\draw[<-] (B)--node [xshift=-5pt, yshift=5pt]{$\alpha_2$}(C);\draw[<-](C)--(D);\draw[loosely dotted, thick](D)--(E);\draw[->] (E)--(F);\draw[<-] (F)--node [xshift=5pt, yshift=5pt]{$\alpha_{s-1}$}(G); \draw[<-] (G)--node [xshift=-5pt, yshift=5pt]{$\alpha_s$}(H);
\draw [->](H)--node [yshift=-5pt]{$\beta$} (A);
\draw [->] (5.6,0.1)--node [yshift=5pt]{$\gamma$} (4.4,0.1);\draw [->] (5.6,-0.1)--node [yshift=-5pt]{$\delta$} (4.4,-0.1);

\end{tikzpicture}}\]
 with the relations $\alpha_p\gamma$ and $\beta\delta$.

They proved in \cite[Theorem 2]{BM} that a degenerate gentle $2$-cycle algebra is derived equivalent to $\Lambda_0(s,r)$ for $s\in\mathbb N^*$ and $0\leq r\leq s-1$ or to $\Lambda'_0(s,0)$ for $s\in\mathbb N^*$ and they conjecture that this list is minimal.

\begin{proposition}\label{propgentle2}
\begin{enumerate}
\item If $s\geq 2$ then the algebra $\Lambda'_0(s,0)$ is a surface algebra associated with the torus with one boundary component with $p=s-1$ marked points on the boundary and with invariant $\gcd(\Lambda)=0$.

\item
If $s\geq 3$ and $0\leq r\leq [\frac{s}{2}]-1$ the algebra $\Lambda_0(s,r)$ is derived equivalent to a surface algebra associated with a torus with one boundary component with $p=s-2$ marked points on the boundary and with invariant $\gcd(\Lambda)=r+1$.
\end{enumerate}
\end{proposition}

\begin{proof}
For the first statement, one  can check that $\Lambda_0'(s,0)$ is the surface alegbra associated with graded triangulation given in figure \ref{figure2}, where the angles of degree $1$ are marked in red.

\begin{figure}
\[\scalebox{0.6}{
\begin{tikzpicture}[scale=1.4,>=stealth]

\draw[fill=gray!20] (0,0) circle (1); 
\node (A00) at (0,1) {$\bullet$};
\node (B00) at (0.5,0.86) {$\bullet$}; 
\node (F00) at (-0.5,0.86) {$\bullet$};
\node at (0.86,0.5) {$\bullet$};
\node at (1,0){$\bullet$};

\node at (6.86,0.5) {$\bullet$};
\node at (7,0){$\bullet$};
\node at (0.86,6.5) {$\bullet$};
\node at (1,6){$\bullet$};
\node at (6.86,6.5) {$\bullet$};
\node at (7,6){$\bullet$};

\draw[thick, blue] (0,1)..  node (B)[fill=white, inner sep=1pt]{$2$}controls(0,1) and (-2.2,6.5)..(-0.5,6.86);
\draw[thick, blue] (0,1).. node (A)[fill=white, inner sep=1pt]{$1$}controls (-1,1) and (-3,7)..(0,7);
\draw[thick, blue] (0,1)..  node (C)[fill=white, inner sep=1pt]{$s-1$}controls(0,1) and (2.2,6.5)..(0.5,6.86);
\draw[thick, blue] (0,1).. node (D)[fill=white, inner sep=1pt, yshift=20pt]{$s$}controls (1,1) and (3,7)..(0,7);

\draw[thick, blue] (0,1)--node (G)[fill=white, inner sep=1pt]{$s+1$}(6,1);
\draw[thick, blue] (0,7)--node (F)[fill=white, inner sep=1pt]{$s+1$}(6,7);
\draw[thick, blue] (0,1)..node (E)[fill=white, inner sep=1pt]{$s+2$}controls (2,1) and (3,7)..(6,7);
\draw[thick, blue] (6,1).. node (H)[fill=white, inner sep=1pt]{$1$}controls (5,1) and (3,7)..(6,7);

\draw[blue, very thick,  dotted] (C)--(B);

\draw[thick,->] (A)--(B);
\draw[thick, ->] (C)--(D);
\draw[thick,->] (F)--(D);
\draw[thick, ->] (E)--(F);
\draw[thick,->] (E)--(G);
\draw[thick, ->] (G)--(H);

\draw[green, very thick, ->] (-2,2)--node [fill=white, inner sep=1pt]{$a$}(8,2);
\draw[green, very thick, ->] (2,-2)--node [fill=white, inner sep=1pt]{$b$}(2,8);
\draw[thick, red] (0.6,1.5)--(0.97,1.5);
\draw[thick, red] (4,6)--(4.6,6);

\draw[fill=gray!20] (6,0) circle (1);
\node (A60) at (6,1) {$\bullet$};
\node (B60) at (6.5,0.86) {$\bullet$}; 
\node (F60) at (5.5,0.86) {$\bullet$};
\draw[fill=gray!20] (0,6) circle (1);
\node (A06) at (0,7) {$\bullet$};
\node (B06) at (0.5,6.86) {$\bullet$}; 
\node (F06) at (-0.5,6.86) {$\bullet$};
 \draw[fill=gray!20] (6,6) circle (1);
\node (A66) at (6,7) {$\bullet$};
\node (B66) at (6.5,6.86) {$\bullet$}; 
\node (F66) at (5.5,6.86) {$\bullet$};

\end{tikzpicture}}\]
\caption{A graded triangulation associated to the algebra $\Lambda'_0(s,0)$ .} \label{figure2}
\end{figure}

One easily checks that for the curves $a$ and $b$ as in Figure \ref{figure2} we have $d(a)=d(b)=0$, so $\gcd(\Lambda)=0$.

\medskip

Using successively \cite[Lemma1.1]{BM}, one shows that the algebra $\Lambda_0(s,r)$ is derived equivalent to the algebra with the same quiver and with the relations: $\alpha_s\beta$, $\gamma\alpha_1$ and $\alpha_{2i}\alpha_{2i+1}$ for $i=1,\ldots, r$. This algebra is the surface algebra associated with the  graded triangulation given in figure~\ref{figure1}, where the angles of degree $1$ are marked in red.

\begin{figure}
\[\scalebox{0.6}{
\begin{tikzpicture}[scale=1.4,>=stealth]

\draw[fill=gray!20] (0,0) circle (1); 
\node (A00) at (0,1) {$\bullet$};
\node at (0.3,0.95) {$\bullet$};
\node (B00) at (0.5,0.86) {$\bullet$}; \node (C00) at (0.5,-0.86){$\bullet$};\node (D00) at (0,-1){$\bullet$};
\node (E00) at (-0.5,-0.86) {$\bullet$};
\node (F00) at (-0.5,0.86) {$\bullet$};
\node at (0.86,0.5) {$\bullet$};
\node at (1,0){$\bullet$};

\node at (6.86,0.5) {$\bullet$};
\node at (7,0){$\bullet$};
\node at (0.86,6.5) {$\bullet$};
\node at (1,6){$\bullet$};
\node at (6.86,6.5) {$\bullet$};
\node at (7,6){$\bullet$};

\draw[blue, thick] (B00) arc (15:135:0.3);
\draw [blue, thick] (D00) arc (225:345:0.3);
\node [blue] at (0.2,-0.8) {$3$};
\node [blue] at (0.2,0.7) {$2r+1$};

\draw[blue, thick] (0,1)-- (0,5);
\draw[blue, thick] (0.5,0.86)..node (B)[fill=white, inner sep=1pt]{$2r$}controls (0.5,0.86) and (0,4)..(0,5);
\draw[blue, thick] (0,-1).. node [fill=white, inner sep=1pt]{$2$}controls (3.5,-2) and (1,4)..(0,5);
\draw[blue, thick] (0.5, -0.86)..node (C) [fill=white, inner sep=1pt]{$4$} controls (2.5,-1.5) and (0.5,4)..(0,5);
\draw[thick, blue] (0,1)--(-0.5,5.14);
\draw[thick, blue] (0,1)..  node (A)[fill=white, inner sep=1pt]{$s-1$}controls(0,1) and (-2.2,6.5)..(-0.5,6.86);
\draw[thick, blue] (0,1).. node [fill=white, inner sep=1pt]{$s$}controls (-1,1) and (-3,7)..(0,7);

\draw[blue, very thick,  dotted] (A)--(-0.25,3.07);

\draw[blue, very thick,  dotted] (B)--(C);

\draw[green, very thick, ->] (-2,2)--node [fill=white, inner sep=1pt]{$a$}(8,2);
\draw[green, very thick, ->] (2,-2)--node [fill=white, inner sep=1pt]{$b$}(2,8);
\draw[thick, red] (0,4)..controls (1,4) and (1.5,4.7)..(1.5,4.6);
\draw[thick, red] (3.9,4.5)--(4.4,4.5);

\draw[thick, blue] (6,1).. node [fill=white, inner sep=1pt]{$s$}controls (5,1) and (3,7)..(6,7);

\draw[thick, blue] (0,-1)..  node [fill=white, inner sep=1pt]{$s+1$}controls (2,-3) and (4,2)..(6,1);
\draw[thick,blue] (0,5)..  node [fill=white, inner sep=1pt]{$s+1$}controls (2,3) and (4,8)..(6,7);

\draw[thick, blue] (0,-1).. node [fill=white, inner sep=1pt]{$1$} controls (5,-3) and (2.5,8).. (6,7);

\draw[fill=gray!20] (6,0) circle (1);
\node (A60) at (6,1) {$\bullet$};
\node at (6.3,0.95) {$\bullet$};
\node (B60) at (6.5,0.86) {$\bullet$}; \node (C60) at (6.5,-0.86){$\bullet$};\node (D60) at (6,-1){$\bullet$};
\node (E60) at (5.5,-0.86) {$\bullet$};
\node (F60) at (5.5,0.86) {$\bullet$};
\draw[blue, thick] (B60) arc (15:135:0.3);
\draw [blue, thick] (D60) arc (225:345:0.3);

\draw[fill=gray!20] (0,6) circle (1);
\node (A06) at (0,7) {$\bullet$};
\node at (0.3,6.95) {$\bullet$};
\node (B06) at (0.5,6.86) {$\bullet$}; \node (C06) at (0.5,5.14){$\bullet$};
\node (D06) at (0,5){$\bullet$};
\node (E06) at (-0.5,5.14) {$\bullet$};
\node (F06) at (-0.5,6.86) {$\bullet$};
\draw[blue, thick] (B06) arc (15:135:0.3);
\draw [blue, thick] (D06) arc (225:345:0.3);

 \draw[fill=gray!20] (6,6) circle (1);
\node (A66) at (6,7) {$\bullet$};
\node at (6.3,6.95) {$\bullet$};
\node (B66) at (6.5,6.86) {$\bullet$}; \node (C66) at (6.5,5.14){$\bullet$};\node (D66) at (6,5){$\bullet$};
\node (E66) at (5.5,5.14) {$\bullet$};
\node (F66) at (5.5,6.86) {$\bullet$};
\draw[blue, thick] (B66) arc (15:135:0.3);
\draw [blue, thick] (D66) arc (225:345:0.3);
\end{tikzpicture}}\]
\caption{A graded triangulation associated to the algebra $\Lambda_0(s,r)$ .} \label{figure1}
\end{figure}

One then checks that for the closed curves $a$ and $b$ marked in green on the figure we have $d(a)=-(r+1)$ and $d(b)=0$ so $\gcd(\Lambda)=r+1$.

\medskip

\end{proof}

Then applying Theorem \ref{thmgcd} one gets the following consequence.
\begin{corollary}\label{cor bobinski}
The algebras of the list $\Lambda_0(s,r)$ $s\geq 3, \ r\leq [\frac{s}{2}-1]$ and $\Lambda_0'(s,0)$ $s\geq 2$ are  pairewise non derived equivalent.
\end{corollary}

\begin{remark} Proposition \ref{prop values} may suggest that the algebras $\Lambda_0(s,r)$ for $s\geq 3$ and  $[\frac{s}{2}]+1\leq r\leq s-1$, and for $s=2$, are not derived equivalent to a surface algebra. For instance, one can easily see that $\Lambda_0(2,0)$ has global dimension $2$ but is not a surface algebra. 
\end{remark}

\begin{remark}
In a subsequent paper \cite{Bob} Bobinski proved that all algebras $\Lambda_0(s,r)$ are pairwise non derived equivalent using Corollary \ref{cor bobinski}. 
\end{remark}

From Proposition \ref{propgentle2} one also deduces the following (compare with Proposition \ref{prop values}).
\begin{corollary}\label{cor values} Let $p$ and $0\leq r\leq \frac{p+1}{2}$ be integers. Then there exists a graded triangulation $(\Delta,d)$ of the torus with one boundary component and $p$ marked points on $\partial S$ such that the invariant of the corresponding surface algebra is $\gcd(\Lambda)=r$. 
\end{corollary}

\end{document}